\documentclass[11pt]{amsart}

\usepackage[T1]{fontenc}
\usepackage[utf8]{inputenc}
\usepackage{textcomp}
\usepackage{lmodern}
\usepackage{microtype}

\usepackage{amssymb}
\usepackage{amsthm}
\usepackage{amscd}
\usepackage{mathabx}

\usepackage{hyperref}

\usepackage{mathscinet}
\usepackage[giveninits=true,maxbibnames=10,doi=false,url=false,sortcites,backend=biber]{biblatex}
\usepackage{tikz}
\usepackage{tikz-3dplot}
\usetikzlibrary{decorations.pathreplacing}
\usetikzlibrary{decorations.pathmorphing}
\usetikzlibrary{patterns}

\newtheorem{theorem}{Theorem}[section]

\newtheorem{lemma}[theorem]{Lemma}

\newtheorem{question}[theorem]{Question}

\theoremstyle{definition}
\newtheorem{definition}[theorem]{Definition}

\def \dom{\operatorname{dom}}
\def \rng{\operatorname{rng}}

\allowdisplaybreaks[2]

\newcommand{\ADSSTS}{requirement}
\newcommand{\SCAC}{$\mathbf{SCAC}$-requirement}
\newcommand{\SPROD}{$\mathbf{SProdWQO}$-requirement}

\begin{document}

\title{Constructing Sequences One Step at a Time}
\author{Henry Towsner}
\date{\today}
\thanks{Partially supported by NSF grant DMS-1600263.\\The work was begun while the author was visiting the Institute for Mathematical Sciences, National University of Singapore in 2016. The visit was supported by the Institute.}
\address {Department of Mathematics, University of Pennsylvania, 209 South 33rd Street, Philadelphia, PA 19104-6395, USA}
\email{htowsner@math.upenn.edu}
\urladdr{\url{http://www.math.upenn.edu/~htowsner}}

\begin{abstract}
We propose a new method for constructing Turing ideals satisfying principles of reverse mathematics below the Chain-Antichain Principle ($\mathbf{CAC}$).  Using this method, we are able to prove several new separations in the presence of Weak K\"onig's Lemma ($\mathbf{WKL}$), including showing that $\mathbf{CAC}+\mathbf{WKL}$ does not imply the thin set theorem for pairs, and that the principle ``the product of well-quasi-orders is a well-quasi-order'' is strictly between $\mathbf{CAC}$ and the Ascending/Descending Sequences principle, even in the presence of $\mathbf{WKL}$.  
\end{abstract}

\maketitle

\section{Introduction}

\begin{definition}
  A \emph{Turing ideal} is a collection $\mathcal{I}$ of sets such that whenever $X\in\mathcal{I}$ and the set $Y$ is computable from $X$, also $Y\in\mathcal{I}$, and whenever $X_1,X_2\in\mathcal{I}$, the join $X_1\oplus X_2\in\mathcal{I}$ as well.
\end{definition}

The principles we discuss here are usually formulated in the context of reverse mathematics, but since that formulation will not be needed here, we state them in terms of Turing ideals.  (Those familiar with reverse mathematics \cite{simpson99} will recognize that our main concern is constructing $\omega$-models witnessing various separations.)  We are interested in Turing ideals which exhibit certain closure properties: ideals $\mathcal{I}$ so that whenever $X\in\mathcal{I}$ encodes an \emph{instance} of problem a certain kind, $\mathcal{I}$ also contains some $Y$ which is a \emph{solution} to that instance.

An important example is:
\begin{definition}
  A Turing ideal $\mathcal{I}$ satisfies $\mathbf{WKL}$ (``Weak K\"onig's Lemma'') if whenever $T\in\mathcal{I}$ encodes an infinite tree of $\{0,1\}$ sequences, there is an infinite $\{0,1\}$ sequence $\Lambda\in\mathcal{I}$ so that for every $n$, $\Lambda\upharpoonright n\in T$.
\end{definition}

\begin{definition}
  We say that a principle $\mathbf{P}$ \emph{implies} $\mathbf{Q}$ if any Turing ideal satisfying $\mathbf{P}$ also satisfies $\mathbf{Q}$.
\end{definition}

All our other principles concern weakenings or variants of Ramsey's Theorem for pairs.  Recall that Ramsey's Theorem for pairs says that whenever $c:[\mathbb{N}]^2\rightarrow\{0,1\}$ is a coloring of pairs, there is an infinite homogeneous set: an infinite set $S\subseteq\mathbb{N}$ and an $i$ so that whenever $a,b\in S$, $c(a,b)=i$

Most of the weakenings we are interested in concern partial or total orders.  An ordering $\prec$ can be associated with a coloring by setting $c(a,b)=1$ iff $a\prec b$ (where we assume $a,b$ are ordered $a<b$ in the usual ordering on the natural numbers).

\begin{definition}
 A Turing ideal $\mathcal{I}$ satisfies $\mathbf{CAC}$ (``Chain-Antichain'') if whenever $\preceq$ is an $\mathcal{I}$-computable partial ordering, there is an infinite sequence $\Lambda$ in $\mathcal{I}$ which is either $\prec$-increasing, $\prec$-decreasing, or an antichain in $\prec$.
\end{definition}
This is equivalent to restricting Ramsey's Theorem for pairs to the special case where one of the colors is transitive \cite{MR2298478}.
\begin{definition}
  If $c:[\mathbb{N}]^2\rightarrow\mathbb{N}$ is a coloring, we say a color $i$ is \emph{transitive} if whenever $a_0<a_1<a_2$ with $c(a_0,a_1)=c(a_1,a_2)=i$, also $c(a_0,a_2)=i$.
\end{definition}

A natural further restriction is to ask that $\preceq$ be a linear ordering.
\begin{definition}
  A Turing ideal $\mathcal{I}$ satisfies $\mathbf{ADS}$ (``Ascending/Descending Sequences'') if whenever $\prec$ is an $\mathcal{I}$-computable linear ordering, there is an infinite sequence $\Lambda$ in $\mathcal{I}$ which is either $\prec$-increasing or $\prec$-decreasing.
\end{definition}
This is slightly stronger than requiring that both colors be transitive, but is equivalent at the level of Turing ideals.

\begin{definition}
  A Turing ideal $\mathcal{I}$ satisfies $\mathbf{trRT^2_k}$ (``transitive Ramsey's Theorem for pairs with $k$ colors'') if whenever $c:[\mathbb{N}]^2\rightarrow[1,k]$ is a coloring where all colors are transitive, there is an infinite set $S$ and an $i\in[1,k]$ so that whenever $a,b\in S$, $c(a,b)=i$.
\end{definition}

The basic relationships between $\mathbf{CAC}$, $\mathbf{ADS}$, and $\mathbf{trRT^2_k}$ are set out in \cite{MR2298478}.

\begin{lemma}[\cite{MR2298478}]
A Turing ideal satisfies $\mathbf{ADS}$ iff it satisfies $\mathbf{trRT^2_2}$.  
\end{lemma}

Furthermore, $\mathbf{CAC}$ implies $\mathbf{trRT^2_k}$ for any $k$, and $\mathbf{trRT^2_{k+1}}$ implies $\mathbf{trRT^2_k}$.

Showing that these implications do not reverse is more difficult.  Lerman, Solomon, and Towsner constructed a Turing ideal satisfying $\mathbf{ADS}$ but not $\mathbf{CAC}$ \cite{LST:MR3125903}, and Patey showed that a similar method can construct a Turing ideal satisfying $\mathbf{trRT^2_k}$ but not $\mathbf{CAC}$ \cite{MR3535177}.  (More precisely, Patey studies a principle shown to be very similar in \cite{MR3219047}.)  It is not known whether $\mathbf{trRT^2_k}$ implies $\mathbf{trRT^2_{k+1}}$.

Dzhafarov, Goh, and Shore asked whether these separations remain in the presence of $\mathbf{WKL}$.  As we will discuss in detail below, satisfying $\mathbf{WKL}$ appears to conflict with the method used in \cite{LST:MR3125903}, and a new approach to the separation is required.  Using this approach, we will show:
\begin{theorem}\label{thm:tr_vs_CAC}
  There is a Turing ideal satisfying $\mathbf{trRT^2_k}$ for all $k$ and $\mathbf{WKL}$ but not $\mathbf{CAC}$.
\end{theorem}

While considering this question, one naturally considers what else might be a consequence of $\mathbf{ADS}$ together with $\mathbf{WKL}$.  In particular, one asks whether these principles might imply other consequences of Ramsey's Theorem for pairs which do not follow from $\mathbf{CAC}$.  For example:
\begin{definition}
A Turing ideal $\mathcal{I}$ satisfies $\mathbf{TS(2)}$ (``Thin Sets for Pairs'') if whenever $c:[\mathbb{N}]^2\rightarrow\mathbb{N}$ is $\mathcal{I}$-computable function, there is an infinite set $S$ in $\mathcal{I}$ and a color $i$ so that there is no $x,y\in S$ with $c(x,y)=i$.
\end{definition}
This thin set principle was introduced in \cite{friedman_thinset} and further studied in \cite{MR2185429,patey_thinset,MR3213294}.

Using a similar method, we are able to show:
\begin{theorem}
  There is a Turing ideal satisfying $\mathbf{CAC}$ and $\mathbf{WKL}$ but not $\mathbf{TS(2)}$.
\end{theorem}

Hirschfeldt and Shore ask \cite{MR2298478} whether the $\mathbf{trRT^2_k}$ hierarchy is strict.
\begin{question}
  Does $\mathbf{trRT^2_k}$ imply $\mathbf{trRT^2_{k+1}}$?
\end{question}
Normally adding more colors does not change the difficulty of satisfying a Ramsey theoretic principle: one ``merges'' two of the colors into a single color and then applies the Ramsey theoretic argument repeatedly.  But this fails with $\mathbf{trRT^2_k}$ because the merged color may not be transitive.

Asking how we should strengthen the statement to allow such a merger of colors leads us to define:
\begin{definition}
A Turing ideal $\mathcal{I}$ satisfies $\mathbf{ProdWQO}$ (``Products of WQOs are WQO'') if whenever $c:[\mathbb{N}]^2\rightarrow\{0,1,2\}$ and the colors $1$ and $2$ are transitive, there is an infinite set $S$ and an $i\in\{1,2\}$ so that whenever $a,b\in S$, $c(a,b)\neq i$.
\end{definition}
(The name will be justified below.)  That is, we have a coloring with two transitive colors and one color which need not be transitive where we can always \emph{omit} one of the transitive colors.

\begin{lemma}[\cite{MR2078917}]
$\mathbf{CAC}$ implies $\mathbf{ProdWQO}$.
\end{lemma}

Frittaion, Marcone, and Shafer pointed out that $\mathbf{ProdWQO}$ implies $\mathbf{ADS}$.
\begin{lemma}
$\mathbf{ProdWQO}$ implies $\mathbf{trRT^2_k}$ for any $k$, and so also $\mathbf{ADS}$.
\end{lemma}
\begin{proof}

  Let $c:[\mathbb{N}]^2\rightarrow [1,k]$ be a transitive coloring.  For any pair $i\neq j$ in $[1,k]$, define the coloring $c_{i,j}:[\mathbb{N}]^2\rightarrow[0,1,2]$ given by 
\[c_{i,j}(a,b)=\left\{\begin{array}{cc}
1&\text{if }c(a,b)=i\\
2&\text{if }c(a,b)=j\\
0&\text{otherwise}
\end{array}\right.\]
By $\mathbf{ProdWQO}$ applied to $c_{1,2}$, we have an infinite set $S$ omitting either color $1$ or color $2$; without loss of generality, we assume $S$ omits $1$.  Applying $\mathbf{ProdWQO}$ to $c_{2,3}$ (more precisely, let $\pi:\mathbb{N}\rightarrow S$ be the unique injective, order-preserving map, define $c'_{2,3}(i,j)=c_{2,3}(\pi(i),\pi(j))$, and apply $\mathbf{ProdWQO}$ to $c'_{2,3}$) restricted to the set $S$, we omit a second color.  We iterate this until only one color is remaining, at which point the set must be homogeneous.
\end{proof}

Although we phrase it here in terms of transitive colorings, $\mathbf{ProdWQO}$ is more naturally seen as the statement that a product of well-quasi-orders is also well-quasi-ordered.  Recall that a partial ordering $\preceq$ is well-quasi-ordered if whenever $\langle a_1,a_2,\ldots\rangle$ is an infinite sequence, there exist $i<j$ so that $a_i\preceq a_j$.  An infinite sequence $\langle a_1,a_2,\ldots\rangle$ is \emph{bad} if it witnesses the failure to be a well-quasi-order: whenever $i<j$, $a_i\not\preceq a_j$.

The product $\preceq=\preceq_1\times\preceq_2$ of two quasi-orderings is given by $a\preceq b$ iff both $a\preceq_1 b$ and $a\preceq_2 b$.  To say that the product of two well-quasi-orders is also well-quasi-ordered is the same as saying that whenever we have a product $\preceq=\preceq_1\times\preceq_2$ and an infinite bad sequence in $\preceq$ then we must have an infinite bad sequence in either $\preceq_1$ or in $\preceq_2$.  If we define a coloring 
\[c(a,b)=\left\{\begin{array}{cc}
1&\text{if }a\preceq_1 b\\
2&\text{if }a\preceq_2 b\\
0&\text{otherwise}
\end{array}\right.\]
then this is well-defined on an infinite bad sequence (because we cannot have both $a\preceq_1 b$ and $a\preceq_2 b$).  The colors $1$ and $2$ are transitive while $0$ need not be.  Finding a bad sequence in $\preceq_i$ exactly means finding an infinite sequence avoiding $i$, which is precisely what our formulation of $\mathbf{ProdWQO}$ says.

Our remaining results show that $\mathbf{ProdWQO}$ is properly intermediate between $\mathbf{ADS}$ and $\mathbf{CAC}$.
\begin{theorem}\ 
  \begin{itemize}
  \item There is a Turing ideal satisfying $\mathbf{trRT^2_k}$ for all $k$ and $\mathbf{WKL}$ but not $\mathbf{ProdWQO}$.
  \item There is a Turing ideal satisfying $\mathbf{ProdWQO}$ and $\mathbf{WKL}$ but not $\mathbf{CAC}$.
  \end{itemize}
\end{theorem}
Of course, either of these results implies Theorem \ref{thm:tr_vs_CAC}.

Finally, we note that all these principles have a stable version.
\begin{definition}
  A coloring of pairs $c:[\mathbb{N}]^2\rightarrow\mathbb{N}$ is \emph{stable} if for every $a$ there are $i$ and $j$ so that whenever $j\leq b$, $c(a,b)=i$.

$\mathbf{SADS}$ (respectively $\mathbf{SCAC}$, $\mathbf{STS(2)}$, $\mathbf{SProdWQO}$, $\mathbf{StrRT^2_k}$) is the principle $\mathbf{ADS}$ (respectively $\mathbf{CAC}$, $\mathbf{TS(2)}$, $\mathbf{ProdWQO}$, $\mathbf{trRT^2_k}$) restricted to stable instances.
\end{definition}
In fact, all our results also apply to the stable versions of these principles; that is, when we show that we fail to satisfy a principle, we always fail to satisfy a stable instance.

The author is grateful to Frittaion, Marcone, and Shafer for pointing out that $\mathbf{ProdWQO}$ is between $\mathbf{ADS}$ and $\mathbf{CAC}$ and raising the question of where it fits.  Some of the ideas leading to the work here were developed in discussions with Kuyper, Lempp, Miller, and Soskova.  Finally, Patey provided feedback and suggestions on a long strong of initial attempts at this work, including pointing the author towards the crucial obstacles and suggesting several ways that the results in this paper could be strengthened.

\section{Separating $\mathbf{STS(2)}$}

In this section we construct a computable instance $c$ of $\mathbf{STS}(2)$ and then construct a Turing ideal $\mathcal{I}$ which has no solution to $c$, but does satisfy both $\mathbf{CAC}$ and $\mathbf{WKL}$.

Since this is the prototype for our other arguments, we take a moment to outline the structure.  The ideal $\mathcal{I}$ will be defined by recursively building a sequence $I_1, I_2,\ldots$ of sets and taking $\mathcal{I}$ to be those things computable from $\oplus_{i\leq n}I_i$ for some $n$.  Given $X=\oplus_{i\leq n}I_i$ for some $n$, we will define the notion of a \emph{\ADSSTS{}} (computable) in $X$, and the notion of when a particular instance $c$ of $\mathbf{STS}(2)$ \emph{satisfies} a given \ADSSTS{} in an oracle $X$.  We will then prove:
\begin{enumerate}
\item if $c$ satisfies all \ADSSTS{}s in $X$ then there is no $X$-computable solution to $c$ (Lemma \ref{thm:ads_sts_diat}),
\item if $c$ satisfies all \ADSSTS{}s in $X$ and $\preceq$ is an $X$-computable partial ordering then there is an infinite chain or antichain $\Lambda$ so that $c$ satisfies all \ADSSTS{}s in $X\oplus\Lambda$ (Lemma \ref{thm:ads_sts_cac}),
\item if $c$ satisfies all \ADSSTS{}s in $X$ and $U$ is an infinite $X$-computable $\{0,1\}$-branching tree then there is an infinite branch $\Lambda$ so that $c$ satisfies all \ADSSTS{}s in $X\oplus\Lambda$ (Lemma \ref{thm:ads_sts_wkl}), and
\item there exists a computable stable $c$ satisfying all \ADSSTS{}s in $\emptyset$ (Lemma \ref{thm:ads_sts_exists}).
\end{enumerate}

These four pieces give the desired result:
\begin{theorem}\label{thm:ads_sts}
  There is a computable stable $c:[\mathbb{N}]^2\rightarrow\mathbb{N}$ and a Turing ideal $\mathcal{I}$ so that:
  \begin{itemize}
  \item if $I\in \mathcal{I}$ is infinite then $c\upharpoonright [I]^2=\mathbb{N}$,
  \item $\mathcal{I}$ satisfies $\mathbf{CAC}$, and
  \item $\mathcal{I}$ satisfies $\mathbf{WKL}$.
  \end{itemize}
\end{theorem}
\begin{proof}
 We take the $c$ given by Lemma \ref{thm:ads_sts_exists} and then use Lemma \ref{thm:ads_sts_cac} and Lemma \ref{thm:ads_sts_wkl} to recursively define the sets $I_i$ so that $c$ satisfies all \ADSSTS{}s in $\oplus_{i\leq n}I_i$, so that if $\preceq$ is an $\oplus_{i\leq n}I_i$-computable partial ordering the there is some $k$ so that $I_k$ is an infinite chain or antichain, and so that if $U$ is an infinite $\oplus_{i\leq n}I_i$-computable $\{0,1\}$-branching tree then there is some $k$ so that $I_k$ is an infinite branch of $U$.  Then the Turing ideal consisting of all sets computable from $\oplus_{i\leq n}I_i$ for some $n$ will have the desired properties.
\end{proof}

\subsection{Requirements}

\begin{definition}
  Let $c:[\mathbb{N}]^2\rightarrow\mathbb{N}$ be stable.  For each $i$, $A^*_i(c)$ consists of those $n$ so that, for cofinitely many $m$, $c(n,m)=i$.
\end{definition}
Clearly the $A^*_i(c)$ are disjoint; stability implies that they form a partition of $\mathbb{N}$.

\begin{definition}
 A \emph{simple block statement in $X$} is a set computable from an oracle $X$ of the form $K^X(b,\vec a)$ (with the groups of variables distinguished) which is monotone in the second parameter---that is, $K^X(b,\vec a')$ and $\vec a'\subseteq \vec a$ implies $K^X(b,\vec a)$.
\end{definition}
The parameters are intended as follows:
\begin{itemize}
\item $b$ is an auxiliary datum,
\item $\vec a$ is a set of witnesses which might be in $A^*_i(c)$ for some $i$.
\end{itemize}

\begin{definition}
A \emph{\ADSSTS} $R=(T,\{K_\sigma\}_{\sigma\in T},\{d_\sigma\}_{\sigma\in T})$ is a finite, finitely branching tree $T$, for each $\sigma\in T$ a simple block statement $K_\sigma$ and a function $d_\sigma:\dom(\sigma)\rightarrow\mathbb{N}$, and so that $K_{\langle\rangle}$ is always true.

For any $\sigma\in T$, any $c:[\mathbb{N}]^2\rightarrow\mathbb{N}$, and any oracle $X$, the \emph{positive requirement component at $\sigma$} is the formula $\Delta^X_{R;\sigma}(c,b_0,\ldots,b_{|\sigma|-1},\vec a_0,\ldots,\vec a_{|\sigma|-1})$ which holds if, for each $i< |\sigma|$, $K^X_{\sigma\upharpoonright(i+1)}((b_0,\ldots,b_{i}),\vec a_i)$ holds.

If $\sigma\in T$ is a leaf, $\Theta^X_{R;\sigma}(c)$ is the formula which holds if there exist $b_0,\ldots,b_{|\sigma|-1},\vec a_0,\ldots,\vec a_{|\sigma|-1}$ so that:
\begin{itemize}
\item $\vec a_i\in A^*_{d_\sigma(i)}(c)$,
\item $\Delta^X_{R;\sigma}(c,b_0,\ldots,b_{|\sigma|-1},\vec a_0,\ldots,\vec a_{|\sigma|-1})$ holds.
\end{itemize}

If $\sigma\in T$ is not a leaf, $\Theta^X_{R;\sigma}(c)$ is the formula which holds if there exist $b_0,\ldots,b_{|\sigma|-1},\vec a_0,\ldots,\vec a_{|\sigma|-1}$ and a $t$ so that:
\begin{itemize}
\item $\vec a_i\in A^*_{d_\sigma(i)}(c)$,
\item $\Delta^X_{R;\sigma}(c,b_0,\ldots,b_{|\sigma|-1},\vec a_0,\ldots,\vec a_{|\sigma|-1})$,
\item there do not exist $b, \vec a$, and $\tau$ an immediate extension of $\sigma$ in $T$ so that $t<\vec a$ and $\Delta^X_{R;\tau}(c,b_0,\ldots,b_{|\sigma|-1},b,\vec a_0,\ldots,\vec a_{|\sigma|-1},\vec a)$.
\end{itemize}


We say $c$ \emph{satisfies a \ADSSTS{} $R=(T,\{K_\sigma\}_{\sigma\in T},\{d_\sigma\}_{\sigma\in T})$ in $X$} if there is some $\sigma\in T$ so that $\Theta^X_{R;\sigma}(c)$ holds.
\end{definition}

We will sometimes wish to work with requirements satisfying certain restrictions.
\begin{definition}
  A \ADSSTS{} $R=(T,\{K_\sigma\}_{\sigma\in T},\{d_\sigma\}_{\sigma\in T})$ has \emph{range $I$} if for every $\sigma\in T$, $\rng(d_\sigma)\subseteq I$.

  A \ADSSTS{} $R=(T,\{K_\sigma\}_{\sigma\in T},\{d_\sigma\}_{\sigma\in T})$ is \emph{transitive in color $i$} if whenever $\tau\sqsubsetneq\sigma$, $j<|\tau|$, $d_\tau(j)=i$, and $d_\sigma(|\tau|)=i$, then $d_\sigma(j)=i$.
\end{definition}

While we mostly find it natural to work with trees of requirement, we note that it does suffice to consider linear ones.
\begin{definition}
  A \ADSSTS{} is \emph{linear} if $\sigma\in T$ implies $\sigma$ has the form $\langle 0,0,\ldots,0\rangle$.
\end{definition}

\begin{lemma}\label{thm:linearize}
  Suppose $c$ satisfies every linear \ADSSTS{} in $X$ with range $I$ which is transitive in every color in $J\subseteq I$ where $0\in I\setminus J$.  Then $c$ satisfies every \ADSSTS{} in $X$ with range $I$ which is transitive in every color in $J$.
\end{lemma}
\begin{proof}
  Let $R=(T,\{K_\sigma\}_{\sigma\in T},\{d_\sigma\}_{\sigma\in T})$ be a \ADSSTS{} with range $I$ which is transitive in every color in $J\subseteq I$.  We define a linear requirement whose satisfaction ensures that we have satisfied $R$.

  Let $n=|T|$ and fix a function $\pi:T\rightarrow [0,n)$ so that $\sigma\sqsubseteq\tau$ implies $\pi(\sigma)\leq\pi(\tau)$.  We let $T'$ consist of sequence of the form $\langle 0,\ldots,0\rangle$ with length $<n$ and we associate the sequence in $T'$ of length $i$ with the natural number $i$.

  When $j<|\sigma|$, we set $d_{\pi(\sigma)}(\pi(\sigma\upharpoonright j))=d_\sigma(j)$, and $d_{\pi(\sigma)}(j)=0$ otherwise.  This ensures that $T'$ will have the same range and satisfy the same transitivity requirements, as needed.

  The auxiliary data will have the form $(r_i,b_i)$ where $r_i$ is either an element of $T$ or $0$.  $(K')^X_i(((r_0,b_0),\ldots,(r_{i-1},b_{i-1})),\vec a_0,\ldots,\vec a_{i-1})$ holds if, letting $i_1,\ldots,i_k<i$ be those values such that $r_{i_j}\neq 0$:
  \begin{itemize}
  \item $k\geq 0$ (i.e. there is at least one such $i$ with $r_{i_j}\neq 0$),
  \item $r_{i_j}$ is a sequence with $|r_{i_j}|=j$,
  \item $r_{i_0}\sqsubsetneq r_{i_1}\sqsubsetneq\cdots\sqsubsetneq r_{i_k}$,
  \item if $0< j<k$ then $i_{j+1}=\pi(r_{i_j})$,
  \item $K^X_{r_{i_k}}((b_{i_0},\ldots,b_{i_k}),\vec a_{i_0},\ldots,\vec a_{i_k})$,
  \item $\pi(r_{i_k})\geq i$.
  \end{itemize}

Suppose $\Theta^X_{R';i}(c)$ holds for some $i$.  Let $\sigma=\pi^{-1}(i)$, and let $(r_0,b_0)$, $\ldots$, $(r_{i-1},b_{i-1})$, $\vec a_0$, $\ldots$, $\vec a_{i-1}$ be the witnessing data.  Let $i_1,\ldots,i_k<i$ be the witnesses; note that if $\pi(r_{i_k})>i$ then we would also satisfy $\Theta^X_{R';i+1}(c)$, so we may assume either $i=0$ (so $\sigma=\langle\rangle$) or $\sigma=r_{i_k}$.  So for any $\tau\sqsubseteq\sigma$, we have $\tau=r_{i_{|\tau|}}$, so $K^X_\tau((b_{i_0},\ldots,b_{i_{|\tau|}}),\vec a_{i_{|\tau|}})$.

On the other hand, if there were some immediate extension $\sigma$ of $r_{i_k}$, a $b$, and a $\vec a$ so that $\Delta^X_{R;\sigma}(c,b_{i_0},\ldots,b_{i_k},b,\vec a_{i_0},\ldots,\vec a_{i_k},\vec a)$ holds then $(\sigma,b),\vec a$ would witness $(K')^X_{i+1}$.  So we have $\Theta^X_{R;r_{i_k}}(c)$.
\end{proof}

\begin{lemma}\label{thm:ads_sts_diat}
  Suppose $c$ satisfies every \ADSSTS{} in $X$ with range $I$ which is transitive in every color in $J\subseteq I$ where $0\in I\setminus J$.  Then whenever $B$ is an $X$-computable (or even $X$-computably enumerable) infinite set, $c\upharpoonright[B]^2\supseteq I$.
\end{lemma}
\begin{proof}
For each $e$ and each $i\in I$, we show that if $W_e$ is infinite then there is an $x\in W_e\cap A^*_i(c)$; then since $W_e$ is infinite, there must be a big enough $x\in W_e$ with $c(x,y)=i$.

We take $T$ to contain a single branch of length $1$, $\langle 0\rangle$.  We take $K_{\langle 0\rangle}^X(b,x)$ to hold if $x\in W^X_{e,b}$.  We set $d_{\langle 0\rangle}(0)=i$.  If $\Theta_{R;\langle\rangle}(c)$ holds then there must be some $t$ so that there do not exist $b$ and $x>t$ so that $x\in W^X_{e,b}$; but this implies that $W_e$ is finite.  Otherwise $\Theta_{R;\langle 0\rangle}(c)$ holds, in which case we find $b_0,x_0$ so that $x_0\in W^X_{e,b_0}$ and $x_0\in A^*_i(c)$ as needed.
\end{proof}

Before going on, we attempt to motivate our definition of a requirement.  Our discussion will be most meaningful to someone already familiar with the construction in \cite{LST:MR3125903}.  For purposes of this discussion, we consider a separation easier than any of the others considered in this paper: separating $\mathbf{ADS}$ from $\mathbf{D}^2_2$; the latter is $\mathbf{STS}(2)$ restricted to the colors $\{0,1\}$, where a solution must omit one of these colors (and therefore be homogeneous in the other color).

We imagine that we are simultaneously constructing our instance $c$ of $\mathbf{D}^2_2$ and our solution to some instance $\prec$ of $\mathbf{ADS}$, and we wish to make a single step of our construction, which means arranging progress towards either a $\prec$-increasing sequence $\Lambda^+$ so that $\Phi_{e_0}^{\Lambda^+}$ fails to compute a solution to $c$ or a $\prec$-decreasing sequence $\Lambda^-$ so that $\Phi_{e_1}^{\Lambda^-}$ fails to compute a solution to $c$.  The key idea of \cite{LST:MR3125903} was to look for both a $\prec$-increasing sequence $p$ with endpoint $p^+$ and a $\prec$-decreasing sequence $q$ with endpoint $q^+$ so that:
\begin{itemize}
\item $p^+\preceq q^+$,
\item there are two fresh elements $a_0,a_1$ so that $\Phi_{e_0}^p$ converges and equals $1$ on both $a_0$ and $a_1$,
\item there are two fresh elements $b_0,b_1$ so that $\Phi_{e_1}^q$ converges and equals $1$ on both $b_0$ and $b_1$.
\end{itemize}
If this happens, we could restrain $c$ so that we will have $a_0,b_0\in A^*_0(c)$ and $a_1,b_1\in A^*_1(c)$.  Then, since $p^+\preceq q^+$, \emph{either} there are infinitely many $x$ with $p^+\prec x$ (and therefore $p$ is a reasonable beginning of an increasing sequence), or there are infinitely many $x$ with $x\prec q^+$ (and therefore $q$ is a reasonable beginning of a decreasing sequence).  Crucially, if we fail to find such a pair $p,q$, then one can arrange for either $\Phi_{e_0}^{\Lambda^+}$ or $\Phi_{e_1}^{\Lambda^-}$ to be finite.

The difficult point is that one needs to ensure $b_0\neq a_1$ and $a_0\neq b_1$ so that we can place both of the needed restraints separately.  

This is the source of the conflict when one attempts to strengthen the separation by including solutions to $\mathbf{WKL}$.  One ends up working not with a single attempt at building $\Lambda^+$ and $\Lambda^-$, but with a finitely branching tree of attempts.  The problem is that even if one finds such pairs $p,q$ in each branch, there may be incompatibilities across different branches --- $a_0$ in one branch may be $b_1$ in another.

What one would prefer is to construct our witnesses in stages.  First we would look for a pair $p_0,q_0$ with $p_0^+\prec q_0^+$ and only the witnesses $a_0,b_0$.  Then we could look for extensions $p_0\sqsubseteq p_1$ and $q_0\sqsubseteq q_1$ with $p_1^+\preceq q_1^+$, and demand that the witnesses $a_1,b_1$ be above some threshold based on the first stage (in particular, larger than $\max\{a_0,b_0\}$).  Such a construction would be compatible with a finitely branching tree: we could wait for the pairs $p_0,q_0$ to appear in every branch.  The witnesses $a_0,b_0$ taken over all branches would form a ``block'' which is all restrained in the same way (say, all put into $A^*_0(c)$).  Only then would we look for the extensions $p_1,q_1$ in all branches, requiring that the witnesses $a_1,b_1$ all be larger than any element of the $0$ block.

The difficulty is that we need the following property: suppose we find our witnesses $p_0,q_0$, but then are unable to extend to $p_1,q_1$.  Then this must be a situation in which we can succeed (presumably by forcing one of $\Phi_{e_0}^{\Lambda^+},\Phi_{e_1}^{\Lambda^-}$ to be finite), even if a different choice of $p_0,q_0$ could have been extended to a $p_1,q_1$.

Let us state this more explicitly, since it is the driving force behind our definition above.  When we wish to satisfy some requirement, we will proceed in stages in which we look for auxiliary data (like $p_0,q_0$) and witnesses (like $a_0,b_0$).  When we find the data and witnesses, we may ``restrain'' the witnesses (by placing them in some $A_i^*(c)$) and then begin looking for the next stage of the construction.  However:
\begin{itemize}
\item during each stage, all witnesses found at a given earlier stage must be restrained the same way, and
\item at each stage, failing to find the data and witnesses to the next stage must be sufficient to ensure our requirement.
\end{itemize}
This is essentially what our definition of satisfaction of a \ADSSTS{} says.

In fact, the two-stage construction we alluded to two paragraphs ago fails: having found the witnesses $p_0,q_0$, failing to find $p_1,q_1$ is not helpful.  It could be that, say, $p_0^+$ will actually turn out to be quite large in $\prec$, and no further elements will appear above $p_0^+$, making the extension $p_1$ impossible to find, and also meaning that our inability to find it gives us no information about how to restrain $\Lambda^+$ to make $\Phi_{e_0}^{\Lambda^+}$ finite.

In Figure \ref{fig:ADS} we lay out a multi-stage process which is substantially more complicated (the version there involves as many as six consecutive steps)  For example, the next stage after finding $p_0,q_0$ is to look for \emph{either} a pair $p_1,q'_0$ with $p_0\sqsubseteq p_1$, $p_1^+\preceq (q'_0)^+$, $p_1$ finds a witness $a_1$, and $q'_0$ finds a new witness $b'_0$, \emph{or} a pair $p'_0,q_1$ with $q_0\sqsubseteq q_1$, $(p'_0)^+\preceq q_1^+$, $q_1$ finds a witness $b_1$, and $p'_0$ finds a new witness $a'_0$.

\begin{figure}
\begin{tikzpicture}[scale=0.6, font=\tiny]
\begin{scope}[shift={(-1,0)}]
  \draw (-1.5,-.5)--(1.45,-.5)--(1.45,.7)--(-1.5,.7)--cycle;
  \draw[->](-1,0)--(-.1,.5); \draw[->](1,0)--(.1,.5);
  \node at (-.5,0) {$p_1$}; \node at (.5,0) {$q_1$};
\end{scope}
\begin{scope}[shift={(2.5,-1.3)}]
  \draw (-1.7,-.5)--(3,-.5)--(3,.7)--(-1.7,.7)--cycle;
  \draw[->] (-1,0)--(0,.25); \draw[dashed,->] (0,.25)--(.9,.5);
  \draw[dashed,->](2,0)--(1.1,.5);
  \node at (0,0) {$p_1$}; \node at (1.5,0) {$q_1$};
\end{scope}
\begin{scope}[shift={(2.5,-3)}]
  \draw (-1.7,-.5)--(3,-.5)--(3,.7)--(-1.7,.7)--cycle;
  \draw[dotted,->](-1,0)--(-.1,.5); 
  \draw[dashed,->](2,0)--(1,.25);\draw[dotted,->](1,.25)--(.1,.5);
  \node at (-.5,0) {$p_2$}; \node at (.75,0) {$q_2$};
\end{scope}
\begin{scope}[shift={(-5,-2.6)}]
  \draw (-1.7,-.5)--(3,-.5)--(3,.7)--(-1.7,.7)--cycle;
  \draw[dashed,->](-1,0)--(-.1,.5); 
  \draw[->](2,0)--(1,.25);\draw[dashed,->](1,.25)--(.1,.5);
  \node at (-.5,0) {$p_2$}; \node at (.5,0) {$q_2$};
\end{scope}
\begin{scope}[shift={(-6,-4.6)}]
  \draw (-1.7,-.5)--(5.2,-.5)--(5.2,.7)--(-1.7,.7)--cycle;
  \draw[dashed, ->](-1,0)--(-.1,.5); 
  \draw[->](2,0)--(1,.25);\draw[dashed, ->](1,.25)--(.1,.5);
  \draw[dotted, ->] (2.75,0)--(3.65,.5);\draw[dotted,->] (4.75,0)--(3.85,.5);
  \node at (-.5,0) {$p_2$}; \node at (1,0) {$q_2$};
  \node at (3.25,0) {$p_1$}; \node at (4.25,0) {$q_1$};
\end{scope}
\begin{scope}[shift={(-7,-6.5)}]
  \draw (-2.7,-.5)--(3,-.5)--(3,.7)--(-2.7,.7)--cycle;
  \draw[dashed,->](-2,0)--(-1,.25); \draw[decorate, decoration={snake,amplitude=.5mm},->](-1,.25)--(-.1,.5);
  \draw[dotted,->](2,0)--(1,.25);\draw[decorate, decoration={snake,amplitude=.5mm},->](1,.25)--(.1,.5);
  \node at (-.75,0) {$p_2$}; \node at (.75,0) {$q_2$};
\end{scope}
\begin{scope}[shift={(0,-6.5)}]
  \draw (-1.7,-.5)--(6,-.5)--(6,.7)--(-1.7,.7)--cycle;
  \draw[dashed,->](-1,0)--(-.1,.5); 
  \draw[->](2,0)--(1,.25);\draw[dashed,->](1,.25)--(.1,.5);
  \draw[dotted,->] (2.75,0)--(3.75,.25);\draw[decorate, decoration={snake,amplitude=.5mm},->] (3.75,.25)--(4.65,.5);
  \draw[decorate, decoration={snake,amplitude=.5mm},->](5.75,0)--(4.85,.5);
  \node at (-.5,0) {$p_2$}; \node at (1,0) {$q_2$};
  \node at (3.75,0) {$p_1$}; \node at (5,0) {$q_1$};
\end{scope}
\begin{scope}[shift={(2,-8.5)}]
  \draw (-2.7,-.5)--(3,-.5)--(3,.7)--(-2.7,.7)--cycle;
  \draw[dashed,->](-2,0)--(-1,.25); \draw[decorate, decoration={zigzag,segment length=1.5mm,amplitude=.5mm},->](-1,.25)--(-.1,.5);
  \draw[decorate, decoration={snake,amplitude=.5mm},->](2,0)--(1,.25);\draw[decorate, decoration={zigzag,segment length=1.5mm,amplitude=.5mm},->](1,.25)--(.1,.5);
  \node at (-.75,0) {$p_2$}; \node at (.75,0) {$q_2$};
\end{scope}
\draw(-1,-.5)--(-1,-1.3);
\draw[->](-1,-1.3)--(.8,-1.3);
\draw(-1,-1.3)--(-4.3,-1.3);
\draw[->](-4.3,-1.3)--(-4.3,-1.9);
\draw[->](3.2,-1.8)--(3.2,-2.3);
\draw[->](-4.3,-3.1)--(-4.3,-3.9);
\draw[->](3.2, -3.5)--(3.2,-4.2);
\node at (3.2,-4.3) {$\cdots$};
\draw(-4.45,-5.1)--(-4.45,-5.4);
\draw(-4.45,-5.4)--(2.3,-5.4);
\draw[->](2.3,-5.4)--(2.3,-5.8);
\draw(-4.45,-5.4)--(-7.1,-5.4);
\draw[->](-7.1,-5.4)--(-7.1,-5.8);
\draw[->](2.3,-7)--(2.3,-7.8);
\end{tikzpicture}
\caption{}\label{fig:ADS}
\end{figure}

\subsection{Solving $\mathbf{ADS}$}

As a warm up to dealing with $\mathbf{CAC}$ (and a preview of Lemma \ref{thm:ads_prod_ads}), we first show that we can solve instances of $\mathbf{ADS}$ while preserving \ADSSTS{}s.

As in \cite{LST:MR3125903}, it is convenient to restrict to a certain kind of linear ordering.
\begin{definition}
  A linear ordering $(\mathbb{N},\prec)$ is \emph{stable-ish} if there is a non-empty initial segment $V$ so that $V$ has no maximum under $\prec$ and $\mathbb{N}\setminus V$ has no minimum under $\prec$.
\end{definition}

\begin{lemma}[\cite{LST:MR3125903}]
  If $(\mathbb{N},\prec)$ is not stable-ish then there is an infinite monotone $\prec$-sequence computable from $\prec$.
\end{lemma}
Note that there is no requirement that the set $V$ be computable from $\prec$.

\begin{lemma}\label{thm:ads_sts_ads}
  Suppose $c$ satisfies every \ADSSTS{} in $X$ and $\prec$ is a stable-ish $X$-computable linear ordering.  Then there is a monotone sequence $\Lambda$ so that $c$ satisfies every \ADSSTS{} in $X\oplus\Lambda$.
\end{lemma}
\begin{proof}
Let $V$ witness that $\prec$ is stable-ish.  When $p$ is a monotone sequence, we write $p^+$ for the final element of $p$.

We will force with \emph{conditions}, which are pairs $(p,q)$ where $p$ is a $\prec$-increasing sequence in $\prec$, $q$ is a $\prec$-decreasing sequence in $\prec$, $p^+\in V$, and $q^+\not\in V$.  (This of course implies that $p^+\prec q^+$.  Note that being a condition is generally not $X$-computable, since $V$ need not be $X$-computable.)  We say a condition $(p',q')$ \emph{extends} $(p,q)$ if $p\sqsubseteq p'$ and $q\sqsubseteq q'$.  We say $(p,q)$ \emph{forces $R$ on the increasing side} if whenever $\Lambda$ is an infinite, $\prec$-increasing sequence with $p\sqsubseteq \Lambda$ and $\Lambda\subseteq V$, $c$ satisfies $R$ in $X\oplus\Lambda$.  Similarly, we say $(p,q)$ \emph{forces $R$ on the decreasing side} if whenever $\Lambda$ is an infinite, $\prec$-decreasing sequence with $q\sqsubseteq \Lambda$ and $\Lambda\subseteq V$, $c$ satisfies $R$ in $X\oplus\Lambda$.

It suffices to show:
\begin{quote}
  $(\ast)$ Suppose $R^+$ and $R^-$ are requirements and $(p,q)$ is a condition.  Then there is a condition $(p',q')$ extending $(p,q)$ which either forces $R^+$ on the increasing side or $R^-$ on the decreasing side.
\end{quote}
For suppose we have shown this.  Then we fix a list of requirements $R^+_i,R^-_i$ so that for any pair of requirements $R^+,R^-$, there is an $i$ with $R^+_i=R^+, R^-_i=R^-$.  We construct a sequence $(\langle\rangle,\langle\rangle)=(p_0,q_0), (p_1,q_1),\ldots$ with $(p_{i+1},q_{i+1})$ extends $(p_i,q_i)$, $(p_{2i+1},q_{2i+1})$ either forces $R^+_i$ on the increasing side or $R^-_i$ on the decreasing side, $p_{2i}$ has length $\geq i$, and $q_{2i}$ has length $\geq i$.  Let $\Lambda^+=\bigcup p_i$ and $\Lambda^-=\bigcup q_i$.  If $c$ does not satisfy every \ADSSTS{} in $X\oplus\Lambda^+$ then there is some $R^+$ which it fails to satisfy, and therefore for each $R^-$ there was an $i$ with $R^+_i=R^+, R^-_i=R^-$, and therefore since $(p_{2i+1},q_{2i+1})$ must not have forced $R^+$ on the increasing side, $(p_{2i+1},q_{2i+1})$ forced $R^-$ on the decreasing side, and therefore $\Lambda^-$ satisfies every \ADSSTS{} in $X\oplus\Lambda^-$.

We now show $(\ast)$.  Let a condition $(p,q)$ and requirements $R^+, R^-$ be given.  Let $R^+=(T^+,\{L_{\sigma}\}_{\sigma\in T^+},\{d^+_\sigma\}_{\sigma\in T^+})$ and $R^-=(T^-,\{M_\tau\}_{\tau\in T^-},\{d^-_\tau\}_{\tau\in T^-})$ be given.   We will describe a requirement $R=(T,\{K_\upsilon\}_{\upsilon\in T},\{d_\upsilon\}_{\upsilon\in T})$.  

For bookkeeping reasons, it is convenient to assume that for any $\sigma\in T^+$, $d^+_\sigma(|\sigma|-1)=0$; this is easily arranged: if $\sigma\in T^+$ violates this, modify $R^+$ as follows: insert a child $\sigma^\frown\langle 0\rangle$ so $L_{\sigma^\frown\langle 0\rangle}$ always holds, and wait for this dummy node to set $d^+_{\sigma^\frown\langle 0\rangle}=d^+_\sigma\cup\{(|\sigma|,0)\}$, then take all children $\sigma^\frown\gamma$ and move them to $\sigma^\frown\langle 0\rangle^\frown\gamma$.  Symmetrically, we make the same assumption for $R^-$.

A \emph{split pair} is a pair $(p',q')$ so that $p\sqsubseteq p'$, $q\sqsubseteq q'$, and $(p')^+=(q')^+$.  Note that a split pair need not be a condition, but being a split pair \emph{is} $X$-computable.  Crucially, when $(p',q')$ is a split pair, one of $(p,q')$ and $(p',q)$ must be a condition (depending on whether the common endpoint belongs to $V$).

Let $r=\max\{|\tau|\mid \tau\in T^-\}$.  Each node $\upsilon\in T$ will describe a situation involving a sequence of split pairs
\[(p_r,q_r), (p_{r-1},q_{r-1}), \ldots, (p_1,q_1), (p_0,q_0)\]
with the endpoints in order, so that $p_r^+=q_r^+\prec p_{r-1}^+=q_{r-1}^+\prec\cdots p_1^+=q_1^+$.

More formally: to each non-empty $\upsilon\in T$, we associate, for each $j\leq r$, sequences $\sigma^\upsilon_j\in T^+$ and $\tau^\upsilon_j\in T^-$.  We require that if one of these sequences is empty then the other is as well (in which case the corresponding split pair is understood to be an empty sequence).

The expectation (encoded below in the definition of $K^X_\upsilon$) is that $p_j$ is a witness to $\sigma^\upsilon_j$ and $q_j$ is a witness to $\tau^\upsilon_j$.  We will also require that $|\tau^\upsilon_j|\in\{0,j\}$ for each $j$.

Suppose we have a sequence of split pairs like this and suppose that $p_{j+1}^+\in V$ but $q_j^+\not\in V$.  (Taken literally there may not be such a $j$, but if we correctly handle the case where the $p_{j+1}^+$ or $q_j^+$ do not exist because $p_{j+1}$ or $q_j$ is the empty sequence, we will be able to ensure there is such a $j$.)  We can look for a split pair of extensions: $p_{j+1}\sqsubseteq p_*$ and $q_j\sqsubseteq q_*$ with $p_*^+=q_*^+$, $p_*$ witnessing an extension of $\sigma^\upsilon_{j+1}$ and $q_*$ witnessing an extension of $\tau^\upsilon_j$.  If we cannot find one of these then one of $p_{j+1}$ or $q_j$ is the desired extension to our condition.

If we do find such a split pair $p_*,q_*$, we would have a new sequence of split pairs
\[(p_r,q_r),\ldots,(p_{j+2},q_{j+2}),(p_*,q_*),(p_{j-1},q_{j-1}),\ldots,(p_1,q_1),(p_0,q_0).\]
In this new sequence, we no longer have a split pair indexed by $j$, but $q_{*}$ now witnesses a branch of length $j+1$.

This will be the way we extend nodes: for each node $\upsilon\in T$, we will have one child for each $j,\sigma_*,\tau_*$ where $\sigma_*$ extends $\sigma^\upsilon_{j+1}$ and $\tau_*$ extends $\tau^\upsilon_j$.

Following this rule, we can see that branches in $T$ must be finite   Let $s=\max\{|\sigma|\mid\sigma\in T^+\}$ and assign to each $\upsilon$ the sequence of numbers
\[(s-|\sigma^\upsilon_r|,s-|\sigma^\upsilon_{r-1}|,\ldots,s-|\sigma^\upsilon_1|).\]
(We ignore $|\sigma^\upsilon_0|$ since this value is always $0$.)  An extension of $\upsilon$ increments some $|\sigma^\upsilon_{j+1}|$ by $1$ and resets $|\sigma^\upsilon_j|$ to $0$.  In particular, the associated sequence of numbers always decreases in the lexicographic ordering, so each branch of $T$ must terminate.

Note that when we extend a split pair in this construction, we extend $p_{j+1}$ and $q_j$ and discard $q_{j+1}$ and $p_j$.  In particular, the only time sequences $p_j, q_{j'}$ share a block of witnesses is that each $p_j$ and $q_j$ share their final block of witnesses.  This does not cause any problems\footnote{Here we are using the fact that we do not have any transitivity restrictions, so there is no intereference between blocks of witnesses as long as they are distinct.} because we have required that $d_{\sigma^\upsilon_j}^+(|\sigma^\upsilon_j|-1)=0=d_{\tau^\upsilon_j}^-(|\tau^\upsilon_j|-1)$, so the two requirements agree on what to do with the shared block.  Other than that, each block of witnesses is associated with at most one of the sequences $p_j$ or $q_j$, and so $d_\upsilon$ should just copy the corresponding value of $d^+_{\sigma^\upsilon_j}$ or $d^-_{\tau^\upsilon_j}$.

More precisely, for each $\upsilon$ and each $j$ we will have functions $\pi_j^\upsilon:\dom(\sigma^\upsilon_j)\rightarrow\dom(\upsilon)$ and $\rho^\upsilon_j:\dom(\tau^\upsilon_j)\rightarrow\dom(\upsilon)$.  $\pi_j^\upsilon(i)$ will tell us at which stage in the construction of $\upsilon$ the sequence $p_j$ was extended to get length $i+1$, and $\tau_j^\upsilon(i)$ will tell us at which stage the sequence $q_j$ was extended to get length $i+1$.

We can now gather up the data we need for each $\upsilon\in T$.  To each $\upsilon\in T$ we associate:
\begin{itemize}
\item for each $j\leq r$, sequences $\sigma^\upsilon_j\in T^+$ and $\tau^\upsilon_j\in T^-$ such that:
  \begin{itemize}
  \item $|\tau^\upsilon_j|\in\{0,j\}$ and
  \item $|\sigma^\upsilon_j|=0$ if and only if $|\tau^\upsilon_j|=0$,
  \end{itemize}
\item functions $\pi^\upsilon_j:\dom(\sigma^\upsilon_j)\rightarrow\dom(\upsilon)$ and $\rho^\upsilon_j:\dom(\tau^\upsilon_j)\rightarrow\dom(\upsilon)$ such that:
  \begin{itemize}
  \item if $\pi^\upsilon_j(i)=\rho^\upsilon_{j'}(i')$ then $j=j'$, $i=|\sigma^\upsilon_j|-1$, adn $i'=|\tau^\upsilon_j|-1$,
  \item if $\pi^\upsilon_j(i)=\pi^\upsilon_{j'}(i')$ then $j=j'$ and $i=i'$,
  \item if $\rho^\upsilon_j(i)=\rho^\upsilon_{j'}(i')$ then $j=j'$ and $i=i'$.
  \end{itemize}
\end{itemize}

The base case is $\sigma^{\langle\rangle}_j=\tau^{\langle\rangle}_j=\langle\rangle$ for all $j\leq r$ and therefore $\sigma^\upsilon_j=\tau^\upsilon_j$ are the empty function.

Suppose we have defined these values for $\upsilon\in T$.  Then whenever $j_0< r$, $\sigma_*\in T^+$ is an immediate extension of $\sigma^\upsilon_{j_0+1}$, and $\tau_*\in T^-$ is an immediate extension of $\tau^\upsilon_{j_0}$, there is a node $\upsilon'=\upsilon^\frown\langle (j_0,\sigma_*,\tau_*)\rangle\in T$ with:
\begin{itemize}
\item for $j\not\in\{j_0,j_0+1\}$, $\sigma^{\upsilon'}_j=\sigma^\upsilon_j$, $\tau^{\upsilon'}_j=\tau^\upsilon_j$, $\pi^{\upsilon'}_j=\pi^\upsilon_j$, and $\rho^{\upsilon'}_j=\rho^\upsilon_j$,
\item $\sigma^{\upsilon'}_{j_0}=\tau^{\upsilon'}_{j_0}=\langle\rangle$,
\item $\pi^{\upsilon'}_{j_0}$ and $\rho^{\upsilon'}_{j_0}$ are the empty function,
\item $\sigma^{\upsilon'}_{j_0+1}=\sigma_*$ and $\pi^{\upsilon'}_{j_0+1}=\pi^\upsilon_{j_0+1}\cup\{(|\sigma_*|-1,|\upsilon'|-1)\}$,
\item $\tau^{\upsilon'}_{j_0+1}=\tau_*$ and $\rho^{\upsilon'}_{j_0+1}=\rho^\upsilon_{j_0}\cup\{(|\tau_*|-1,|\upsilon'|-1)\}$.
\end{itemize}
Note the definition of $\pi^{\upsilon'}_{j_0+1}$ and $\rho^{\upsilon'}_{j_0+1}$: these are noting that $|\upsilon|'-1$ is the stage at which we extended $\sigma^{\upsilon'}_{j_0+1}$ and $\tau^{\upsilon'}_{j_0+1}$.

The definition forces us to take:
\begin{itemize}
\item $d_\upsilon(\pi^\upsilon_j(i))=d^+_{\sigma^\upsilon_j}(i)$, and
\item $d_\upsilon(\rho^\upsilon_j(i))=d^-_{\tau^\upsilon_j}(i)$.
\end{itemize}
The almost disjointness of the ranges of the various functions $\pi^\upsilon_j$ and $\rho^\upsilon_j$ ensures that we can satisfy this obligation.  This may not fully define $d_\upsilon$, and we may take other values arbitrarily (these correspond to blocks of witnesses no longer in use, and which are therefore irrelevant).

We next need to specify the block statements $K^X_\upsilon((b_0,\ldots,b_{|\upsilon|-1}),\vec a)$.  Each datum $b_i$ will have the form $(e_i,p_i,f_i,q_i)$, where $p_i$ and $q_i$ are the split pair found at stage $i$, and $e_i,f_i$ are additional data needed to witness the corresponding requirements $L_{\sigma}$ and $M_{\tau}$.

When $\upsilon=\langle\rangle$, there is nothing to specify, so assume $\upsilon=\upsilon_-^\frown\langle j_0,\sigma_*,\tau_*\rangle$.  Then $K^X_\upsilon$ needs to verify that $(p_{|\upsilon|-1},q_{|\upsilon|-1})$ are a split pair, positioned correctly relative to our other split pairs, extending the appropriate sequences built at previous stages, and witnessing $L_{\sigma_*}$ and $M_{\tau_*}$.

When we have the sequence of data $(b_0,\ldots,b_{|\upsilon|}-1)$, we need to extract the subsequences corresponding to $p_i$ and $q_i$.  Define $\hat e=(e_{\pi^\upsilon_{j_0}}(0),\ldots,e_{\pi^\upsilon_{j_0}}(|\sigma_*|-1))$ and $\hat f=(f_{\rho^\upsilon_{j_0}}(0),\ldots,f_{\rho^\upsilon_{j_0}}(|\tau_*|-1))$.  Then we define $K^X_\upsilon((b_0,\ldots,b_{|\upsilon|}-1),\vec a)$ to hold if:
\begin{itemize}
\item $(p_{|\upsilon|-1},q_{|\upsilon|-1})$ is a split pair,
\item for each $j<j_0$ such that $\tau^{\upsilon}_j\neq\langle\rangle$, $p_{|\upsilon|-1}^+\prec q^+_{\rho^\upsilon_j(|\tau^\upsilon_j|-1)}$,
\item for each $j\in(j_0,r]$ such that $\tau^\upsilon_j\neq\langle\rangle$, $q_{\rho^\upsilon_j(|\tau^\upsilon_j|-1)}^+\prec p_{|\upsilon|-1}^+$,
  
\item if $|\sigma^\upsilon_{j_0}|>1$ then $p_{\pi^\upsilon_{j_0}}(|\sigma^\upsilon_{j_0}|-2)\sqsubseteq p_{|\upsilon|-1}$,
\item if $|\sigma^\upsilon_{j_0}|=1$ then $p\sqsubseteq p_{|\upsilon|-1}$,
\item if $|\tau^\upsilon_{j_0}|>1$ then $q_{\rho^\upsilon_{j_0}}(|\tau^\upsilon_{j_0}|-2)\sqsubseteq q_{|\upsilon|-1}$,
\item if $|\tau^\upsilon_{j_0}|=1$ then $q\sqsubseteq q_{|\upsilon|-1}$,
\item $L^{X\oplus p_{|\upsilon|-1}}_{\sigma_*}(\hat e,\hat a)$,
\item $M^{X\oplus q_{|\upsilon|-1}}_{\tau_*}(\hat f,\hat a)$.
\end{itemize}
The second and third requirements ensure that our split pairs are ordered correctly.  The fourth through seventh ensure that we are extending the sequences from the previous stage which we promised to exend.  The final two ensure that we have actually found the promised witnesses to $L_{\sigma_*}$ and $M_{\sigma_*}$.

In particular, the sequence of split pairs we discussed above is given by taking $(p_{\pi^\upsilon_j(|\sigma^\upsilon_j|-1)},q_{\pi^\upsilon_j(|\sigma^\upsilon_j|-1)})$ to be the $j$'th split pair.  We chose chosen $K^X_\upsilon$ so that $\Delta^X_{R;\upsilon}(c,b_0,\ldots,b_{|\upsilon|-1},\vec a_0,\ldots,\vec a_{|\upsilon|-1})$ implies that, for each $j\leq r$ we have
\[\Delta^{X\oplus p_{\pi^\upsilon_j(|\sigma^\upsilon_j|-1)}}_{R^+;\sigma^\upsilon_j}(c,e_{\pi^\upsilon_j(0)},\ldots,e_{\pi^\upsilon_j(|\sigma^\upsilon_j|-1)},\vec a_0,\ldots,\vec a_{|\sigma^\upsilon_j|-1})\]
and
\[\Delta^{X\oplus q_{\rho^\upsilon_j(|\tau^\upsilon_j|-1)}}_{R^-;\tau^\upsilon_j}(c,f_{\rho^\upsilon_j(0)},\ldots,f_{\rho^\upsilon_j(|\tau^\upsilon_j|-1)},\vec a_0,\ldots,\vec a_{|\tau^\upsilon_j|-1}).\]

So suppose that $c$ satisfies $R$ in $X$.  What remains is to show that we have the suitable extension of the original split pair $(p,q)$.  Since $c$ satisfies $R$, we may choose an $\upsilon$ so that $\Theta^X_{R;\upsilon}(c)$ holds.

If there is any $j$ so that $\sigma^\upsilon_j$ and $\tau^\upsilon_j$ are both leaves then, since $(p_{\pi^\upsilon_j(|\sigma^\upsilon_j|-1)},q_{\rho^\upsilon_j(|\tau^\upsilon_j|-1)})$ is a split pair, one of $(p_{\pi^\upsilon_j(|\sigma^\upsilon_j|-1)},q)$ or $(p,q_{\rho^\upsilon_j(|\tau^\upsilon_j|-1)})$ is a condition.  Suppose $(p_{\pi^\upsilon_j(|\sigma^\upsilon_j|-1)},q)$ is a condition; then this condition forces $T^+$ on the increasing side since $\Theta^{X\oplus p_{\pi^\upsilon_j(|\sigma^\upsilon_j|-1)}}_{T^+;\sigma^\upsilon_j}(c)$ holds for the leaf $\sigma^\upsilon_j$.  Similarly, if $(p,q_{\rho^\upsilon_j(|\tau^\upsilon_j|-1)})$ is a condition then this condition forces $T^-$ on the decreasing side.

If $\upsilon$ is a leaf, we claim there must be such a $j$.  Suppose there is no such $j$.  Since $\upsilon$ is a leaf, for each $j<r$, we must have at least one of $\sigma^\upsilon_{j+1}$ or $\tau^\upsilon_j$ is a leaf---otherwise we would have an extension corresponding to some $j,\sigma_*,\tau_*$.  $\tau^\upsilon_r$ is always a leaf (because $r=\max\{|\tau|\mid \tau\in T^+\}$), so if $\sigma^\upsilon_r$ is a leaf, we are done.  If not, $\tau^\upsilon_{r-1}$ must be a leaf.  Again, if $\sigma^\upsilon_{r-1}$ is a leaf, we are done; otherwise $\tau^\upsilon_{r-1}$ must be a leaf.  Continuing in this way, since there is no desired $j$, we conclude that $\tau^\upsilon_0$ must be a leaf.  But $\tau^\upsilon_0=\langle\rangle$, so if this is a leaf then $T^-$ is a trivial requirement and $(p,q)$ already satisfies $T^-$ on the decreasing side.

Suppose $\upsilon$ is not a leaf.  If there is any $j\leq r$ such that $|\tau^\upsilon_j|>0$ and $q_{\rho^\upsilon_j(|\tau^\upsilon_j|-1)}^+\not\in V$, take the largest such $j$ and let $\hat q=q_{\rho^\upsilon_j(|\tau^\upsilon_j|-1)}$.  Otherwise, let $\hat q=q$ and $j=0$.  If $\tau^\upsilon_j$ is a leaf then $(p, \hat q)$ is a condition such that $\theta^{X\oplus \hat q^+}_{T^-;\tau^\upsilon_j}(c)$ holds, so we are done.  So assume $\tau^\upsilon_j$ is not a leaf.  If $|\sigma^\upsilon_{j+1}|>0$ then let $\hat p=p_{\pi^\upsilon_{j+1}(|\sigma^\upsilon_{j+1}|-1)}$, otherwise let $\hat p=p$.  Observe that $\hat p^+\in V$: if $\hat p\neq p$ then $\hat p^+=q_{\rho^\upsilon_{j+1}(|\tau^\upsilon_{j+1}|-1)}^+$ which, by maximality of $j$, belongs to $V$.  So if $\sigma^\upsilon_{j+1}$ is a leaf then $(\hat p,q)$ is an extension by a similar argument.

So consider the case where neither $\sigma^\upsilon_{j+1}$ nor $\tau^\upsilon_j$ are leaves.  Then $(\hat p,\hat q)$ is an extension of $(p,q)$, and we claim it forces either $T^+$ on the increasing side or $T^-$ on the decreasing side.  Suppose not, so there are $\Lambda^+$ and $\Lambda^-$ witnessing this failure: $\Delta^{X\oplus\Lambda^+}_{R^+;\sigma^\upsilon_{j+1}}$ and $\Delta^{X\oplus\Lambda^-}_{R^-;\tau^\upsilon_j}$ hold but $\Theta^{X\oplus\Lambda^+}_{R^+;\sigma^\upsilon_{j+1}}$ and $\Theta^{X\oplus\Lambda^-}_{R^-;\tau^\upsilon_j}$ do not.  Then there must be finite $p^*,q^*$ so that $\hat p\sqsubseteq p^*\sqsubset\Lambda^+$ and $\hat q\sqsubseteq q^*\sqsubset\Lambda^-$ are large enough to witness $L^{X\oplus p^*}_{R^+;\sigma_*}$ and $M^{X\oplus q^*}_{R^-;\tau_*}$ for some $\sigma_*$ and $\tau_*$.  

We need to fix $p^*$ and $q^*$ so they share a common endpoint: choose some $s\in V$ so that, for each $j'>j$, $p_{\pi^\upsilon_{j'}(|\sigma_{j'}^\upsilon|-1)}^+\prec s$ (such an $s$ exists since, for each such $j'$, $p_{\pi^\upsilon_{j'}(|\sigma_{j'}^\upsilon|-1)}^+\in V$).  Then $(p^*{}^\frown\langle s\rangle,q^*{}^\frown\langle s\rangle)$ is a split pair witnessing $K^X_{R;\upsilon^\frown\langle j,\sigma_*,\tau_*\rangle}$.

\end{proof}

Before going on, we note some general features of our constructions illustrated by this argument.  Say we have some requirements $R^+$ and $R^-$, and we are attempting to produce a requirement $R$ so that whenever some $\Theta^X_{R;\upsilon}(c)$ holds, we have either a sequence $p$ so that some $\Theta^{X\oplus p}_{R^+;\sigma}(c)$ holds or some sequence $q$ so that some $\Theta^{X\oplus q}_{R^-;\tau}(c)$ holds.  (In the proof above, $p$ and $q$ were chains, but later they will be antichains or other kinds of sequences.)

When $\Delta^X_{R;\upsilon}(c,b_0,\ldots,b_{|\upsilon|-1},\vec a_0,\ldots,\vec a_{|\upsilon|-1})$ holds, the $b_i$ must encode the description of a list of sequences $p_1,\ldots,p_{r^+}$ and sequences $q_1,\ldots,q_{r^-}$ which are candidates to be the needed witnesses, so that $\Delta^{X\oplus p_j}_{R^+;\sigma^\upsilon_j}(c,\cdots)$ and $\Delta^{X\oplus q_j}_{R^-;\sigma^\upsilon_j}(c,\cdots)$ will hold (with witnesses encoded suitably in the $b_i$).  We will call these the ``witnessing sequences''.

Consider some witnessing sequence $p_j$.  This sequence must have been constructed in $|\sigma^\upsilon_j|$ segments, with each segment corresponding to some stage of $\upsilon$: that is, there should be a function $\pi^\upsilon_j:[0,|\sigma^\upsilon_j|)\rightarrow[0,|\upsilon|)$ so that when $\Delta^X_{R;\upsilon}(c,b_0,\ldots,b_{|\upsilon|-1},\vec a_0,\ldots,\vec a_{|\upsilon|-1})$ holds, this implies that 
\[\Delta^{X\oplus p_j}_{R^+;\sigma^\upsilon_j}(c,e_0,\ldots,e_{|\sigma^\upsilon_j|-1},\vec a'_0,\ldots,\vec a'_{|\sigma^\upsilon_j|-1})\]
holds where each $e_j$ is encoded in $b_{\pi^\upsilon_j(|\sigma^\upsilon_j|-1)}$ and each $\vec a'_j\subseteq \pi_j(|\sigma^\upsilon_j|-1)$.

We can make a crucial observation about the stages at which our witnessing sequences get extended.  Suppose that $\upsilon'$ is some immediate successor of $\upsilon$, and that there is a witnessing sequence $p_{j'}$ at stage $\upsilon'$ with $\pi^{\upsilon'}_{j'}(|\sigma^{\upsilon'}_{j'}|-1)=|\upsilon|$---that is, at stage $\upsilon$ there was a witnessing sequence $p_{j}$ and $p_{j'}$ is a proper immediate extension of it, so $\upsilon$ was one of the stages at which $p_{j'}$ was constructed.  Then we must have had $d_\upsilon(\pi^\upsilon_j(i))=d_{\sigma^{\upsilon}_{j}}(i)$ for all $i<|\sigma^\upsilon_{j}|$.  When this happens, we say $p_{j}$ is \emph{active} at $\upsilon$.  Otherwise we say $p_j$ is \emph{inactive}, and is therefore not eligible to be extended at stage $\upsilon$.

This basic structure, of active and inactive witnessing sequences constructed in stages and the functions $\pi$ (and the parallel functions $\rho$) which correspond stages of $\upsilon$ with stages of $\sigma$ or $\tau$, will appear in all our arguments.

\subsection{Solving $\mathbf{CAC}$}

It is convenient to restrict ourselves to partial orderings which are refinements of the usual ordering on $<$; the following lemma shows that this restriction is harmless for our purposes.
\begin{lemma}
  Suppose $\mathcal{I}$ is a Turing ideal and whenever $\preceq$ is a partial ordering in $\mathcal{I}$ so that $a<b$ implies $b\not\preceq a$, $\mathcal{I}$ contains either an infinite chain or an infinite chain in $\preceq$.  Then $\mathcal{I}$ contains an infinite chain or antichain for every partial ordering.
\end{lemma}
\begin{proof}
  Let $\preceq$ be an arbitrary partial ordering in $\mathcal{I}$.  Define $a\preceq' b$ if $a\leq b$ and $a\preceq b$.  Then $\mathcal{I}$ contains either a chain or an antichain for $\preceq'$; if $\mathcal{I}$ contains a chain then it is also a chain in $\preceq$.  Suppose $n_1<n_2<\cdots$ is an infinite antichain in $\mathcal{I}$.  For $a\leq b$, define $a\preceq^*b$ if $n_b\preceq n_a$.  Then $\preceq^*$ is a partial ordering with a chain or an antichain in $\mathcal{I}$, which is also a chain or antichain for $\preceq$.
\end{proof}

\begin{lemma}\label{thm:ads_sts_cac}
  Suppose $c$ satisfies every \ADSSTS{} in $X$ and $\preceq$ is a partial ordering so that $a<b$ implies $b\not\preceq a$.  Then there is an infinite $\Lambda$ which is either a chain or an antichain so that $c$ satisfies every \ADSSTS{} in $X\oplus\Lambda$.
\end{lemma}
\begin{proof}
  We force with conditions which are triples $(p,q,S)$ so that:
  \begin{itemize}
  \item $p$ is a chain,
  \item $q$ is an antichain,
  \item $S$ is an infinite $X$-computable set, $p<S$, $q<S$, if $a\in p$, $b\in q$, and $c\in S$ then $a\prec c$ and $b\not\prec c$.
  \end{itemize}
A condition $(p',q',S')$ extends $(p,q,S)$ if $p\sqsubseteq p', q\sqsubseteq q'$, $(p'\setminus p)\subseteq S$, $(q'\setminus q)\subseteq S$, and $S'\subseteq S$.  We say $(p,q,S)$ \emph{forces $R$ on the chain side} if whenever $\Lambda$ is an infinite chain extending $p$ with $\Lambda\setminus p\subseteq S$, $c$ satisfies $R$ in $X\oplus \Lambda$.  Similarly, we say $(p,q,S)$ \emph{forces $R$ on the antichain side} if whenever $\Lambda$ is an infinite antichain extending $q$ with $\Lambda\setminus q\subseteq S$, $c$ satisfies $R$ in $X\oplus \Lambda$.

For any $x\in S$, let $S_{\succ x}=\{y\in S\mid x\prec y\}$ and $S_{\perp x}=\{y\in S\mid x<y, x\not\prec y\}$, so $S\setminus (S_{\succ x}\cup S_{\perp x})$ is finite.  Then either $(p^\frown\langle x\rangle,q,S_{\succ x})$ or $(p,q^\frown\langle x\rangle,S_{\perp x})$ is a condition.  In particular, we may always extend at least one of $p$ and $q$ by one element.  Furthermore, if there do not exist at least one $x$ which can be added to the $p$ side and at least one which can be added to the $q$ side then $\preceq$ has an $X$-computable chain or antichain: say there is no $x$ which can be added to the $q$ side, so for every $x\in S$, $S_{\perp x}$ is finite.  Then we can greedily add elements from $S$ to $p$ and obtain an infinite chain.

So it suffices to show:
\begin{quote}
  ($\ast$) Suppose $R^+$ and $R^-$ are requirements and $(p,q,S)$ is a condition.  Then there is a condition $(p',q',S')$ extending $(p,q,S)$ which either forces $R^+$ on the chain side or $R^-$ on the antichain side.
\end{quote}
For suppose we have shown this.  Then we fix a list of requirements $R^+_i,R^-_i$ so that for any pair of requirements $R^+,R^-$, there is an $i$ with $R^+_i=R^+, R^-_i=R^-$.  We construct a sequence $(\langle\rangle,\langle\rangle)=(p_0,q_0,S_0), (p_1,q_1,S_1),\ldots$ with $(p_{i+1},q_{i+1},S_{i+1})$ extends $(p_i,q_i,S_i)$, $(p_{2i+1},q_{2i+1},S_{2i+1})$ either forces $R^+_i$ on the chain side or $R^-_i$ on the antichain side, $p_{2i}$ has length $\geq i$, and $q_{2i}$ has length $\geq i$.  Let $\Lambda^+=\bigcup p_i$ and $\Lambda^-=\bigcup q_i$.  If $c$ does not satisfy every \ADSSTS{} in $X\oplus\Lambda^+$ then there is some $R^+$ which it fails to satisfy, and therefore for each $R^-$ there was an $i$ with $R^+_i=R^+, R^-_i=R^-$, and therefore since $(p_{2i+1},q_{2i+1})$ must not have forced $R^+$ on the chain, $(p_{2i+1},q_{2i+1})$ forced $R^-$ on the antichain, and therefore $\Lambda^-$ satisfies every \ADSSTS{} in $X\oplus\Lambda^-$.

So it suffices to show ($\ast$).  Let $R^+=(T,\{L_\sigma\}_{\sigma\in T+},\{d^+_\sigma\}_{\sigma\in T+})$ and $R^-=(T^-,\{M_\tau\}_{\tau\in T^-},\{d^-_\tau\}_{\tau\in T^-})$.  Let $D=\max\{|\sigma|\mid\sigma\in T^+\}$ and $E=\max\{|\tau|\mid\tau\in T^-\}$.  We will describe a requirement $R=(T,\{K_\upsilon\}_{\upsilon\in T},\{d_\upsilon\}_{\upsilon\in T})$.

We attempt to outline the construction before the proof.  With $\mathbf{CAC}$, we only have the benefit of transitivity for one side of our construction.  The analog of a split pair is a \emph{supported antichain}; this is a tuple $(p^0,\ldots,p^m,q')$ where $q'$ is an antichain built in $m$ segments, $q'=q_0^\frown q_1^\frown\cdots^\frown q_m$, and each $p^i$ is a chain with $(p^i)^+\prec x$ for each $x\in q_i$.  (In this discussion, we always assume that all chains and antichains we discuss are contained in $S$.)


Then any $x$ is incomparable to every element of the chain $q'$, or above some $x$ in some $q_i$, and therefore above the chain $p^i$.  One can think of a split pair as the case where only the last ``support'', $p^m$, needs to be retained.

We will be able to extend $q'$ with a new segment while leaving the $p^i$ intact, as long as we can find a suitable $p^{m+1}$ to support the new segment.  Extending a $p^i$, however, will break the antichain $q$; instead, the extension of $p^i$ will have to involve using the extension to support a new antichain.  This leads to some difficult bookkeeping to keep track of all the supported antichains we need, which we will discuss in detail later.

The need to retain the support $p^i$ when we extend $q$ complicates our construction: it means that $p^i$ and $q_i$ cannot share the same block of witnesses, because when $q$ extends, the restraint on $q^i$'s block of witnesses can change, so $q^i$ cannot be competing with $p^i$ for how to restrain this block of witnesses.  So we will need a mechanism to find antichains $q'$ above chains $p'$ so that the witnesses to $p'$ are in a different block from the witnesses to $q'$.

We can illustrate our approach to this by looking at the simplest case for $\mathbf{CAC}$: we have two requirements of length $1$, say $R^+$ and $R^-$.  For simplicity, let us say $T^+$ and $T^-$ each consist of a single non-empty node with simple block statements $K_+$ and $K_-$.  The larger bookkeeping issues do not interfere.

The corresponding tree $T$ has a single immediate descendent of $\langle\rangle$, say $\langle 1\rangle$.  $K^X_{\langle 1\rangle}$ will demand that we find a block of witnesses $\vec a$ and an antichain $q'$ witnessing $K_-^X$ such that, for every $x\in q'$, we have a chain $p'$ with $x=(p')^+$ witnessing $K^X_+$, with all witnesses coming from the block $\vec a$.

Suppose we cannot find such a $q'$ and such a family of $p'$---that is, suppose $\Theta^X_{T;\langle\rangle}(c)$ holds.  If there is a $t$ so that every antichain $p'$ witnessing $K^X_+$ has $(p')^+\leq t$ then, by forcing with $(p,q,S\cap(t,\infty))$, we have forced $R^+$ on the chain side (by ensuring that $\Theta^{X\oplus\Lambda}_{T^+;\langle\rangle}(c)$ will hold).  On the other hand, if antichains witnessing $K^X_+$ appear unboundedly, we can take $S'\subseteq S$ to consist of those $x$ such that there exists an antichian $p'$ witnessing $K^X_+$ with $(p')^+=x$, and by forcing with $(p,q,S')$, we have forced $R^-$ on the antichain side.

At the stage $\langle 1\rangle$, $q'$ is active and the $p'$ are (potentially) inactive---that is, $d_{\langle\rangle}(0)=d^-_{\langle\rangle}(0)$.

$\langle 1\rangle$ will also have one immediate descendent, $\langle 1,1\rangle$.  $K^X_{\langle 1,1\rangle}$ will look for single chain from our family, $p'$, and a new antichain $q^*$ such that, for every $x\in q^*$, $(p')^+\prec x$.  Suppose we cannot find such a $q^*$, so $\Theta^X_{T;\langle 1\rangle}(c)$ holds.  If $\bigcap_{x\in q'}S_{\perp x}$ is infinite then we can extend to $(p,q', \bigcap_{x\in q'}S_{\perp x})$ and have forced $T^-$ on the antichain side by satisfying $\Theta^{X\oplus q'}_{T^-;\langle 1\rangle}(c)$.  If there are only finitely many such elements, then there must be some infinite set $S'\subseteq S$ and a single one of our chains, $p'$, such that, for every $x\in S'$, $(p')^+\prec x$.  In this case we can extend to $(p,q,S')$ and have forced $T^-$ on the antichain side by ensuring that $\Theta^{X\oplus\Lambda}_{T^-;\langle \rangle}(c)$ will hold.

$\langle 1,1\rangle$ is a leaf.  We make $p'$ and $q^*$ both active, which we can do since their witnesses come from different blocks.  If $\Theta^X_{T;\langle 1,1\rangle}(c)$ holds then one of $S_{\perp (p')^+}$ and $S_{\succ (p')^+}$ is infinite, so either $(p',q,S_{\succ (p')^+})$ or $(p,q^*,S_{\perp (p')^+})$ is the extension we want.

We note that there are two distinct attempts to find an antichain; these form a key part of our construction, so we give them names.  The first attempt, when we construct an antichain $q'$ out of endpoints of chains, we call a \emph{trial antichain}.  The second attempt, when we find $q^*$ above some chain $p'$, will be our strategy for finding supported antichains.

In our full construction, we will have to use this trial antichain construction many times: every time we need to extend a supported antichain, we will need to first (attempt to) construct a trial antichain.  If the construction of the trial antichain fails, we will find the witnesses we need.  If the construction succeeds, we will then be able to look for a suitable segment of a supported antichain.

Now we turn to organizing the many partnered antichains we will need to keep track of.  Let $r=\max\{|\sigma|\mid \sigma\in T^+\}$ and $s=\max\{|\tau|\mid \tau\in T^+\}$.  Each supported antichain has $\leq s$ segments, and each segment is supported by a chain of length $\leq r$.  Our goal is to work towards a supported antichain of length $s$, each of whose segments is supported by a chain of length $r$: in this case all our chains and our antichain must witness leaves of $T^+$ or $T^-$, respectively, and therefore one of them will suffice to extend $(p,q,S)$ by.

The difficulty is keeping track of the chains and antichains so that we can make sure we make progress.  (For instance, it is possible to loop if we extend chains carelessly.)  For every function $\omega:[0,s)\rightarrow(0,r]$, we will have an antichain $q_\omega$ corresponding to $\omega$---that is, our goal is to arrange for $q_\omega$ to have segments with the specified support.  In particular, when $q_\omega$ has length $i$, we will not extend $q_\omega$ unless there is a suitable chain of length $\omega(i)-1$ which we expect will support the next segment of $q_\omega$.

Dually, we will have a collections of chains.  Our chains will be indexed by partial functions.  For each $s'\in[0,s)$, let us write $\overline{s'}$ for the set $[0,s)\setminus\{s'\}$.  Then let $G$ be the set of functions $\gamma$ such that, for some $s_\gamma\in[0,s)$, $\gamma:\overline{s_\gamma}\rightarrow(0,r]$.  Then, for each $\gamma\in G$, we will keep track of a chain $p_\gamma$.  The chain $p_\gamma$ will only ever support the segment $s_\gamma$ of an antichain, and it will only support antichains $q_\omega$ such that $\omega\upharpoonright\overline{s_\gamma}=\gamma$.  (For example, suppose $s=r=2$, and consider $\gamma:(0,2)\rightarrow (0,2)$ with $\gamma(1)=2$.  Then when we construct $p_\gamma$ to have length $1$, it must be because $p_\gamma$ is supporting the first segment of $q_{\{(0,1),(1,2)\}}$.  Later, we might succeed in extending $p_\gamma$ to have length $2$, at which point it must be supporting the first segment of $q_{\{(0,2),(1,2)\}}$.)

For each node $\upsilon\in T$, we will have:
\begin{itemize}
\item for each $\gamma$, a $\sigma^\upsilon_\gamma\in T^+$,
\item a monotone function $\pi^\upsilon_\gamma:\dom(\sigma^\upsilon_\gamma)\rightarrow\dom(\upsilon)$.
\end{itemize}

We say that $\omega$ is \emph{relevant} at $\upsilon$ if, for each $i<|\tau^\upsilon_\omega|$, $|\sigma^\upsilon_{\omega\upharpoonright\overline{i}}|\leq \omega(i)$.  When $\omega$ is not relevant, it means that one of the $\gamma$'s needed to support $\omega$ has ``outgrown'' $\omega$---that is, grown taller than $\omega(i)$---and therefore we can no longer support $\omega$.

For each node $\upsilon$ and each $\omega$ relevant at $\upsilon$, we will keep track of:
\begin{itemize}
\item a $\tau^\upsilon_\omega\in T^-$, and 
\item a monotone function $\rho^\upsilon_\omega:\dom(\tau^\upsilon_\omega)\rightarrow\dom(\upsilon)$.
\end{itemize}
We require that the ranges of the $\pi^\upsilon_\gamma,\rho^\upsilon_\omega$ be pairwise disjoint.  We will always have $d_\upsilon(\pi_\gamma^\upsilon(i))=d^+_{\sigma^\upsilon_\gamma}(i)$ and $d_\upsilon(\rho_\omega^\upsilon(i))=d^-_{\tau^\upsilon_\omega}(i)$.

In order to be able to extend a supported antichain, we must have the right matchup between a relevant antichain and a chain: that is, we need a relevant $\omega$ and a $\gamma$ such that $|\tau^\upsilon_\omega|=s_\gamma$, $\omega\upharpoonright\overline{s_\gamma}=\gamma$, and $|\sigma^\upsilon_\gamma|=\omega(s_\gamma)-1$.  When this happens, we say $\omega$ is \emph{active} at $\upsilon$.

For each $\omega$ which is active at $\upsilon$, we also have:
\begin{itemize}
\item a $\hat \tau^\upsilon_\omega\in T^-$, and
\item a $\hat\rho^\upsilon_\omega:\dom(\tau^\upsilon_\omega)\rightarrow\dom(\upsilon)$
\end{itemize}
representing a trial antichain.  We require that the ranges of the $\hat\rho^\upsilon_\omega$ be pairwise disjoint and be disjoint from the ranges of all $\pi^\upsilon_\omega$ and $\rho^\upsilon_\omega$.  Of course we set $d_\upsilon(\hat\rho^\upsilon_\omega(i))=d^-_{\hat\tau^\upsilon_\omega}(i)$.

This is the full information we need to associate with a node $\upsilon$.  For the base case, we define:
\begin{itemize}
\item for each $\gamma$, $\sigma^{\langle\rangle}_\gamma=\langle\rangle$ and $\pi^{\langle\rangle}_\gamma$ is the empty function,
\item for each $\omega$, $\tau^{\langle\rangle}_\omega=\langle\rangle$ and $\rho^{\langle \rangle}_\omega$ is the empty function.
\end{itemize}
This means that $\omega$ is active exactly when $\omega(0)=1$.  For all such $\omega$, we define:
\begin{itemize}
\item $\hat\tau^{\langle\rangle}_\omega=\langle\rangle$ and $\hat\rho^{\langle\rangle}_\omega$ is the empty function.
\end{itemize}

Given a node $\upsilon$, we describe the children of $\upsilon$.  These children come in two types---the version where we extend a trial antichain and the version where we match up a supported antichain with a partnered chain.

For each active $\omega_0$ and each immediate extension $\tau$ of $\hat\tau^\upsilon_{\omega_0}$ in $T^-$, there is an extension $\upsilon'=\upsilon^\frown\langle (0,\omega_0,\tau)\rangle$ with:
\begin{itemize}
\item for each $\gamma$, $\sigma^{\upsilon'}_\gamma=\sigma^\upsilon_\gamma$ and $\pi^{\upsilon'}_\gamma=\pi^\upsilon_\gamma$,
\item for each $\omega$, $\tau^{\upsilon'}_\omega=\tau^\upsilon_\omega$ and $\rho^{\upsilon'}_\omega=\rho^\upsilon_\omega$,
\item $\hat\tau^{\upsilon'}_{\omega_0}=\tau$ and $\hat\rho^{\upsilon}_{\omega_0}=\{(|\tau|-1,|\upsilon'|-1)\}$,
\item for each active $\omega\neq\omega_0$, $\hat\tau^{\upsilon'}_\omega=\hat\tau^\upsilon_\omega$ and $\hat\rho^{\upsilon'}_\omega=\hat\rho^\upsilon_\omega$.
\end{itemize}

For each active $\omega_0$, let $\gamma_0=\omega_0\upharpoonright \overline{|\tau^\upsilon_{\omega_0}|}$.  For each immediate extension $\sigma$ of $\sigma^\upsilon_{\gamma_0}$, each immediate extension $\tau$ of $\tau^\upsilon_{\omega_0}$, and each $j_0<|\hat\tau^\upsilon_{\omega_0}|$, there is an extension $\upsilon'=\upsilon^\frown\langle (1,\omega_0,\sigma,\tau,j_0)\rangle$ with:
\begin{itemize}  
\item $\sigma^{\upsilon'}_{\gamma_0}=\sigma$ and $\pi^{\upsilon'}_{\gamma_0}=\pi^\upsilon_\gamma\cup\{(|\sigma|-1,\hat\rho^\upsilon_{\omega_0}(j_0))\}$,
\item for each $\gamma\neq\gamma_0$, $\sigma^{\upsilon'}_\gamma=\sigma^\upsilon_\gamma$ and $\pi^{\upsilon'}_\gamma=\pi^\upsilon_\gamma$,
  
\item $\tau^{\upsilon'}_{\omega_0}=\tau$ and $\rho^{\upsilon'}_{\omega_0}=\rho^\upsilon_{\omega_0}\cup\{(|\tau|-1,|\upsilon'|-1)\}$,  
\item for each $\omega\neq\omega_0$ which is still relevant at $\upsilon'$, $\tau^{\upsilon'}_\omega=\tau^\upsilon_\omega$ and $\rho^{\upsilon'}_\omega=\rho^\upsilon_\omega$,
\item for each $\omega$ active at $\upsilon'$, $\hat\tau^{\upsilon'}_\omega=\langle\rangle$ and $\hat\rho^{\upsilon'}_\omega$ is the empty function.
\end{itemize}

To see that this tree is finite, observe that a branch can only have finitely many extensions of the second kind in a row---there are only finitely many choices for $\omega_0$, and each extension of the second kind extends one of the $\hat\tau^\upsilon_{\omega_0}$, which can happen at most $s$ times.  Each extension of the first kind extends one of the $\tau^\upsilon_{\omega_0}$, which can also happen at most $s$ times.

We next need to define the block statements $K^X_\upsilon((b_0,\ldots,b_{|\upsilon|-1}),\vec a)$.  For $\upsilon=\langle\rangle$ this is trivial.  Otherwise, let $\ell=|\upsilon|-1$.

Suppose $\upsilon$ is a node of the first kind, $\upsilon=\upsilon_-^\frown\langle (0,\omega_0,\tau)\rangle$.  Let $\gamma_0=\omega_0\upharpoonright\overline{|\tau^\upsilon_{\omega_0}|}$.  The auxiliary datum $b_\ell$ has the form $(k_\ell,p^0_\ell,e^0_\ell,\ldots,p^{k_\ell}_\ell,e^{k_\ell}_\ell,q_\ell,f_\ell)$.  If $\omega_0(|\tau^\upsilon_{\omega_0}|)=1$ then let $p'=p$; otherwise, let $p'=p_{\pi^\upsilon_{\gamma_0}(|\sigma^\upsilon_{\gamma_0}|-1)}$.  If $|\tau|=1$, let $q'=q$.  Otherwise let $q'=q_{\hat \rho^\upsilon_{\omega_0}(|\hat\tau^\upsilon_{\omega_0}|-1)}$.  Then we define $K^X_\upsilon((b_0,\ldots,b_{|\upsilon|-1}),\vec a)$ to hold if:
\begin{itemize}
\item for each $k\leq k_\ell$, $p'\sqsubseteq p^k_\ell$ and $(p^k_\ell\setminus p')\subseteq S$,
\item for each $k\leq k_\ell$, $p^k_\ell$ is a chain and some $\sigma^k_\ell$ immediately extending $\sigma^\upsilon_{\gamma_0}$ so that $L^{X\oplus p^k_\ell}_{\sigma^k_\ell}((e_{\pi^\upsilon_{\gamma_0}(0)},\ldots,e_{\pi^\upsilon_{\gamma_0}(|\sigma^\upsilon_{\gamma_0}|-1)}),e^k_\ell),\hat a)$ holds,
  \item $q'\sqsubseteq q_\ell$,
  \item for each $x\in q_\ell\setminus q'$, there is a $k\leq k_\ell$ so that $(p^k_\ell)^+=x$ (and therefore $(q_\ell\setminus q')\subseteq S$),
  \item $q_\ell$ is an antichain so that $M^{X\oplus q_\ell}_\tau((f_{\hat\rho^\upsilon_{\omega_0}(0)},\ldots,f_{\hat\rho^\upsilon_{\omega_0}(|\hat\tau^\upsilon_{\omega_0}|-1)},f_\ell),\vec a)$ holds.
  \end{itemize}

  Suppose we have a node of the second kind, $\upsilon=\upsilon_-^\frown\langle (1,\omega_0,\sigma,\tau,j_0)\rangle$ and let $\ell=|\upsilon|-1$.  The auxiliary datum $b_\ell$ has the form $(p_\ell,e_\ell,q_\ell,f_\ell)$.  If $|\tau|=1$, let $q'=q$.  If $|\tau|>1$, let $q'=q_{\rho^{\upsilon_0}_{\omega_0}(|\tau^{\upsilon_0}_{\omega_0}|-1)}$.  Then we define $K^X_\upsilon((b_0,\ldots,b_{|\upsilon|-1}),\vec a)$ to hold if:
\begin{itemize}
\item there is a $k\leq k_{\hat\rho^{\upsilon_0}_{\omega_0}(j_0)}$ with $p_\ell=p^k_{\hat \rho^{\upsilon_0}_{\omega_0}}$, $e_\ell=e^k_{\rho \rho^{\upsilon_0}_{\omega_0}(j_0)}$, and $\sigma=\sigma^k_{\hat\rho^{\upsilon_0}_{\omega_0}(j_0)}$,
\item $q'\sqsubseteq q_\ell$,
\item $(q_\ell\setminus q')\subseteq S$,
\item for each $x\in (q_\ell\setminus q')$, $(p_\ell)^+\prec x$,
\item $M_\tau^{X\oplus q_\ell}((f_{\rho^\upsilon_{\omega_0}(0)},\ldots,f_{\rho^\upsilon_{\omega_0}(|\tau^\upsilon_{\omega_0}|-1)},f_\ell),\vec a)$ holds.
\end{itemize}

These choices are made so that whenever $\Delta_{R;\upsilon}^X(c,b_0,\ldots,b_{|\upsilon_-|},\vec a_0,\vec a_{|\upsilon_-|})$ holds, we have:
\begin{itemize}
\item for each $\gamma$, $\Delta_{R^+;\sigma^\upsilon_\gamma}^{X\oplus p_{\pi^\upsilon_\gamma(|\sigma^\upsilon_\gamma|-1)}}(c,e_{\pi^\upsilon_\gamma(0)},\ldots,e_{\pi^\upsilon_\gamma(|\sigma^\upsilon_\gamma|-1)},\vec a_{\pi^\upsilon_\gamma(0)},\ldots,\vec a_{\pi^\upsilon_\gamma(|\sigma^\upsilon_\gamma|-1)})$ holds,
\item for each $\omega$, $\Delta_{R^-;\tau^\upsilon_\omega}^{X\oplus q_{\rho^\upsilon_\omega(|\tau^\upsilon_\omega|-1)}}(c,f_{\rho^\upsilon_\omega(0)},\ldots,f_{\rho^\upsilon_\omega (|\tau^\upsilon_\omega|-1)},\vec a_{\rho^\upsilon_\omega(0)},\ldots,\vec a_{\rho^\upsilon_\omega(|\tau^\upsilon_\omega|-1)})$ holds,
  
\item for each active $\omega$, $\Delta_{R^-;\hat\tau^\upsilon_\omega}^{X\oplus q_{\hat\rho^\upsilon_\omega(|\hat\tau^\upsilon_\omega|-1)}}(c,f_{\hat\rho^\upsilon_\omega(0)},\ldots,f_{\hat\rho^\upsilon_\omega(|\hat\tau^\upsilon_\omega|-1)},\vec a'_0,\ldots,\vec a'_{|\hat\tau^\upsilon_\omega|-1})$ for some $\vec a'_i\subseteq \vec a_{\hat\rho^\upsilon_\omega(i)}$,
\item for each active $\omega$, each $j<|\hat\tau^\upsilon_\omega|$, and each $k\leq k_{\hat\rho^\upsilon_\omega(j)}$,
  \[\Delta^{X\oplus p^k_{\hat\rho^\upsilon_\omega(j)}}_{R^+;\sigma^k_{\hat\rho^\upsilon_\omega(j)}}(c,e_{\pi^\upsilon_\gamma(0)},\ldots,e_{\pi^\upsilon_\gamma(|\sigma^\upsilon_\gamma|-1)},e^k_{\hat\rho^\upsilon_\omega(j)},\vec a_{\pi^\upsilon_\gamma(0)},\ldots,\vec a_{\pi^\upsilon_\gamma(|\sigma^\upsilon_\gamma|-1)},\vec a_{\hat\rho^\upsilon_\omega(j)})\]
holds.
\end{itemize}

Now suppose that there is some $\upsilon\in T$ so that $\Theta^X_{R;\upsilon}(c)$ holds.  We must find the needed extension of $(p,q,S)$.

First, suppose there is some $\gamma$ so that $\sigma^\upsilon_\gamma$ and $S_{\succ (p^\upsilon_{\pi^\upsilon_\gamma(|\sigma^\upsilon_\gamma|-1)})^+}$ is infinite.  Then $(p_{\pi^\upsilon_\gamma(|\sigma^\upsilon_\gamma|-1)},q, S_{\succ (p^\upsilon_{\pi^\upsilon_\gamma(|\sigma^\upsilon_\gamma|-1)})^+})$ witnesses $R^+$ on the chain side since $\Theta^{X\oplus p_{\pi^\upsilon_\gamma(|\sigma^\upsilon_\gamma|-1)}}_{R^+;\sigma^\upsilon_\gamma}(c)$ must hold.

Similarly, if there is an $\omega$ so that $\tau^\upsilon_\omega$ is a leaf and $\bigcap_{x\in q_{\rho^\upsilon_\omega}(|\tau^\upsilon_\omega|-1)}S_{\perp x}$ is infinite then $(p, q_{\rho^\upsilon_\omega}(|\tau^\upsilon_\omega|-1), \bigcap_{x\in q_{\rho^\upsilon_\omega}(|\tau^\upsilon_\omega|-1)}S_{\perp x})$ witnesses $R^-$ on the antichain side since $\Theta^{X\oplus q_{\rho^\upsilon_\omega}(|\tau^\upsilon_\omega|-1)}_{R^-;\tau^\upsilon_\omega}(c)$ holds.

So suppose that there is no such $\gamma$ and no such $\omega$.  We argue that some $\omega$ must be active.

First, consider any $\omega$ such that $\tau^\upsilon_\omega$ is a leaf.  Then $\bigcap_{x\in q_{\rho^\upsilon_\omega}(|\tau^\upsilon_\omega|-1)}S_{\perp x}$ is not infinite, so for cofinitely many $x$, there is an $i<|\tau^\upsilon_\omega|$ such that $(p_{\pi^\upsilon_{\omega\upharpoonright\overline{i}}(|\sigma^\upsilon_{\omega\upharpoonright\overline{i}}|-1)})^+\prec x$.  Since there are only finitely many such $i$, there is some single $i$ so that, taking $\gamma=\omega\upharpoonright\overline{i}$, $S_{\succ (p_{\pi^\upsilon_\gamma(|\sigma^\upsilon_\gamma|-1)})^+}$ is infinite, and therefore $\sigma^\upsilon_\gamma$ must not be a leaf.

We now look for an active $\omega$.  Consider the function $\omega_0$ which is constantly equal to $r$ (and therefore always relevant).  If $|\tau^\upsilon_{\omega_0}|=s$ then $\tau^\upsilon_{\omega_0}$ is a leaf, and therefore there is an $i<|\tau^\upsilon_{\omega_0}|$ so that $\sigma^\upsilon_{{\omega_0}\upharpoonright\overline{i}}$ is not a leaf.  But $|\sigma^\upsilon_{{\omega_0}\upharpoonright\overline{i}}|=\omega_0(i)=r$, which is a contradiction.

So $|\tau^\upsilon_{\omega_0}|<s$.  If $\omega_0$ is not active, it must be because $|\sigma^\upsilon_{{\omega_0}\upharpoonright\overline{|\tau^\upsilon_{\omega_0}|}}|<r-1$.  So consider $\omega_1$ given by setting $\omega_1(|\tau^\upsilon_{\omega_0}|)=|\sigma^\upsilon_{{\omega_0}\upharpoonright\overline{|\tau^\upsilon_{\omega_0}|}}|+1$ and $\omega_1(i)=\omega_0(i)=r$ for all other $s'$.

If $\omega_1$ is not active, it must be because $|\sigma^\upsilon_{{\omega_1}\upharpoonright\overline{|\tau^\upsilon_{\omega_1}|}}|<r-1=\omega_1(|\tau^\upsilon_{\omega_1}|)-1$, so we can find an $\omega_2$ by the same process.  The length $|\tau^\upsilon_{\omega_k}|$ decreases at each step, so we must eventually find an $\omega$ which is active.

Let $\gamma=\omega\upharpoonright\overline{|\tau^\upsilon_\omega|}$.  Let $q'=q_{\hat\rho^\upsilon_\omega(|\hat\tau^\upsilon_\omega|-1)}$ and let $p'=p_{\pi^\upsilon_\gamma(|\sigma^\upsilon_\gamma|-1)}$.

Suppose that $S'=\bigcap_{x\in q'}S_{\perp x}$ is cofinite, and let $S''$ consist of those $x$ such that there is a $p^*\in S$ with $p'\sqsubseteq p^*$, $(p^*)^+=x$, and such that there is an immediate extension $\sigma$ of $\sigma^\upsilon_\gamma$ in $T^+$ so that there exist witnesses to $K^{X\oplus p^*}_{T^+;\sigma}$.  If $S''$ is finite then $(p',q,S'\setminus S_{\perp x})$ witnesses $R^+$ on the chain side by satisfying $\Delta^{X\oplus p'}_{R^+;\sigma^\upsilon_\gamma}(c)$.

If $S''$ is infinite then $S''$ has an infinite computable subset $S^*$ and $(p,q',S^*)$ witnesses $R^-$ on the antichain side by satisfying $\Delta^{X\oplus q'}_{R^-;\hat\tau^\upsilon_\omega}(c)$.

Otherwise, suppose $S'$ is not cofinite.  Then there must be some $x\in q'$ so that $S_{\succ x}$ is infinite.  There is some $j<|\hat\tau^\upsilon_\omega|$ and some $k\leq k_{\hat\rho^\upsilon_\omega(j)}$ so that $(p^k_{\hat\rho^\upsilon_\omega(j)})^+=x$.  Then $(p,q_{\rho^\upsilon_\omega(|\tau^\upsilon_\omega|-1)},S_{\succ x})$ satisfies $R^-$ on the antichain side by satisfying $\Delta^{X\oplus q_{\rho^\upsilon_\omega(|\tau^\upsilon_\omega|-1)}}_{R^-;\tau^\upsilon_\omega}(c)$.
\end{proof}

\subsection{Solving $\mathbf{WKL}$}

We wish to show:
\begin{lemma}\label{thm:ads_sts_wkl}
  Suppose $c$ satisfies every \ADSSTS{} in $X$ and $U_e$ is an infinite, $\{0,1\}$-branching, $X$-computable tree.  Then there is an infinite path $\Lambda$ so that $c$ satisfies every \ADSSTS{} in $X\oplus\Lambda$.
\end{lemma}

We will need variants of this repeatedly, so we state and prove a mild generalization, essentially showing that the same holds if we place various restrictions on the kinds of \ADSSTS{}s we wish to deal with.

\begin{lemma}\label{thm:ads_sts_wkl_gen}
  Let $J\subseteq I\subseteq\mathbb{N}$ be given with $0\in I\setminus J$.  Suppose $c$ satisfies every \ADSSTS{} in $X$ with range $I$ which is transitive in every $j\in J$ and $U_e$ is an infinite, $\{0,1\}$-branching, $X$-computable tree.  Then there is an infinite path $\Lambda$ so that $c$ satisfies every \ADSSTS{} in $X\oplus\Lambda$ with range $I$ which is transitive in every $j\in J$.
\end{lemma}
Then Lemma \ref{thm:ads_sts_wkl} is the case with $I=\mathbb{N}$ and $J=\emptyset$.
\begin{proof}
  By Lemma \ref{thm:linearize}, it suffices to show that for any linear requirement $R=(T,\{K_\sigma\},\{d_\sigma\})$, we can find an initial segment $\lambda\in U_e$ and an infinite $X$-computable $U'\subseteq U_e$ of extensions of $\lambda$ so that whenever $\Lambda$ is a branch through $U_e$, $c$ satisfies $R$ in $X\oplus\Lambda$.

  We will describe a requirement $R'=(T',\{L_\sigma\},\{d'_\sigma\})$ with range $I$ which is transitive in every $j\in J$.  $R'$ will share the same tree, $T'=T$.

  The auxiliary datum $b_i$ will have the form $(s_i,k_i,b^1_i,\ldots,b^{k_i}_i)$ where $s_i$ is a suitable bound, $k_i$ is the number of branches we need to consider, and the $b^j_i$ are the corresponding data for $K^X_\upsilon$.

  $L^X_\upsilon((b_0,\ldots,b_{|\upsilon|-1}),\vec a)$ will hold if, for every $\lambda\in U_e$ with $|\lambda|=s_i$, there is a sequence $j_i\leq k_i$ so that $K^{X\oplus\lambda}_\upsilon((b^{j_0}_0,\ldots,b^{j_{|\upsilon|-1}}_{|\upsilon|-1}),\vec a)$ holds.
  
  This means that when $\Delta^X_{R';\upsilon}$ holds, each $\lambda$ satisfies $\Delta^{X\oplus\lambda}_{R;\upsilon}$.

  Naturally we have $d'_\upsilon=d_\upsilon$, which ensures that $d'_\upsilon$ is transitive.

  We must check that satisfaction of our requirement ensures that we can choose a $\lambda$ forcing satisfaction of the original requirement.  Suppose we satisfy $\Theta^X_{R';\upsilon}(c)$.  Consider the tree $U''\subseteq U_e$ consisting of those $\lambda'$ such that $\Delta^{X\oplus\lambda'}_\upsilon$ holds but we cannot find witnesses to $\Delta^{X\oplus\lambda'}_{\upsilon^\frown\langle0\rangle}$ which extend the fixed witnesses to $\Delta^{X\oplus\lambda'}_\upsilon$.  If $U''$ were finite then we would satisfy $\Delta^X_{R';\upsilon^\frown\langle 0\rangle}$, so $U''$ is infinite, and there must be some $\lambda$ satsifying $\Delta^{X\oplus\lambda}_\upsilon$ with infinitely many extensions in $U''$.  Letting $U'\subseteq U''$ consist of the extensions of $\lambda$, we have forced $\Theta^{X\oplus\Lambda}_{R;\upsilon}(c)$.
\end{proof}

\subsection{Constructing $\mathbf{STS}$(2)}

\begin{lemma}\label{thm:ads_sts_exists}
  There is a computable stable $c:[\mathbb{N}]^2\rightarrow\mathbb{N}$ satisfying all \ADSSTS{}s in $\emptyset$.
\end{lemma}

Again, we prove a more general version that will include later cases.
\begin{lemma}\label{thm:ads_sts_exists_gen}
  Let $J\subseteq I\subseteq\mathbb{N}$ with $0\in I\setminus J$.    There is a computable stable $c:[\mathbb{N}]^2\rightarrow I$ transitive in every color in $J$ and satisfying all \ADSSTS{}s in $\emptyset$ with range $I$ which are transitive in every color in $J$.
\end{lemma}
Again, Lemma \ref{thm:ads_sts_exists} is the case with $J=\emptyset$ and $I=\mathbb{N}$.
\begin{proof}
This is a standard finite injury priority argument.  Informally, we place all \ADSSTS{}s with range $I$ transitive in every color in $J$ in order, and every time we find witnesses violating a negative requirement component, we remember the witnesses, restrain them so future colors comply with the corresponding positive requirement component, and injure all lower priority requirements; that requirement is then witnessed along a longer branch $\upsilon$.  Since each requirement has a finite tree, each requirement eventually stops acting, either because some negative requirement component holds or because we reach a leaf.

More formally, we proceed as follows.  We order the \ADSSTS{}s $R_0,R_1,\ldots$.  At each stage $s$ we have fixed:
\begin{itemize}
\item $c_s:[s]^2\rightarrow I$ transitive in each color in $J$,
\item for $r<s$, $\upsilon_{s,r}\in T_r$, $b_{s,r,0},\ldots,b_{s,r,|\upsilon_{s,r}|-1}$, $\vec a_{s,r,0},\ldots,\vec a_{s,r,|\upsilon_{s,r}|-1}$, $t_{s,r}$, and sets $A_{s,r,j}$ so that:
  \begin{itemize}
  \item for each $i<|\upsilon_{s,r}|$, $K_{\upsilon_{s,r}\upharpoonright(k+1)}((b_{s,r,0},\ldots,b_{r,i}),\vec a_{s,r,i})$,
  \item if $j\neq j'$ then $A_{s,r,j}\cap A_{s,r',j'}=\emptyset$,
  \item each $\vec a_{s,r,i}\in A_{s,r,d_{\upsilon_{s,r}}(i)}$,
  \item if $r'<r$ then $t_{s,r'}\leq t_{s,r}$ and $t_{s,r'}<\vec a_{s,r,i}$,
  \item if $b\in A_{s,r,i}$, $i\in J$, $a<b$, and $c(a,b)=i$ then $a\in A_{s,r,i}$.
  \end{itemize}
 \end{itemize}
We will have $c_s\subseteq c_{s+1}$.  The sets $\bigcup_{r\leq s}A_{s,r,i}$ are approximations to $A^*_i(c)$.  If $a\not\in\bigcup_{r\leq s}\bigcup_i A_{s,r,i}$, we will treat $a$ as if it belongs to some $A_{s,r,0}$.

Suppose we have constructed up to stage $s$.  Define $c_{s+1}(n,s+1)$ for $n<s+1$ by setting $c_{s+1}(n,s+1)=i$ if $n\in A_{s,r,i}$ for some $r$.  (The closure condition on $A_{s,r,i}$ ensures transitivity of $c$.)  Let $r< s$ be least (if there is any) so that there is some $b$, some $\vec a\in(t_{s,r},s+1)$, and some $\upsilon$ an immediate extension of $\upsilon_{s,r}$ in $T_r$ so that $K_{\upsilon}((b_{s,r,0},\ldots,b_{s,r,|\upsilon_{s,r}|-1},b),\vec a)$ holds; otherwise $r=s$.  For $r'<r$, we have $\upsilon_{s+1,r'}=\upsilon_{s,r'}$, $b_{s+1,r',i}=b_{s,r',i}$, $\vec a_{s+1,r',i}=\vec a_{s,r',i}$, $t_{s+1,r'}=t_{s,r'}$, and $A_{s+1,r',i}=A_{s,r,i}$.

If $r<s$, let $\upsilon_{s+1,r}=\upsilon$, $b_{s+1,r,|\upsilon|-1}=b$, $b_{s+1,r,i}=b_{s,r,i}$, $\vec a_{s+1,r,|\upsilon|-1}=\vec a$, $\vec a_{s+1,r,i}=\vec a_{s,r,i}$, and $t_{s+1,r}=s+1$.  Take $A_{s+1,r,j}$ to consist of those $\vec a_{s+1,r,i}$ with $d_{\upsilon}(i)=j$, together with any elements required by the closure condition.  Note that if $a\in A_{s+1,r',j}$ for some $r'<r$ then $c_s(a,b)=j$ for any $b>t_{s,r'}$, so in particular any $\vec a_{s+1,r,i}$, so if $b\in A_{s+1,r,j}$, there is no conflict with having $a\in A_{s+1,r,j}$ as well.

For $r'\in (r,s]$ (or $r'=s$ if $r=s$), set $\upsilon_{s+1,r'}=\langle\rangle$, $A_{s+1,r',i}=\emptyset$, $t_{s+1,r'}=s+1$, and $A_{s+1,r',j}=\emptyset$.


We only injure a requirement $R_j$ if we make the node $\upsilon_{s,j'}$ longer for some $j'<j$, so a requirement is injured only finitely many times.  In particular, there is a limiting node $\upsilon_{j}=\lim_s \upsilon_{s,j}$.  The witnesses $b_{s,j,0},\ldots,b_{s,j,|\upsilon_j|-1}$ and $\vec a_{s,j,0},\ldots,\vec a_{s,j,|\upsilon_j|-1}$ also stabilize to witnesses $b_{j,0},\ldots,b_{j,|\upsilon_j|-1}$ and $\vec a_{j,0},\ldots,\vec a_{j,|\upsilon_j|-1}$.  In particular, these witness $\Delta_{R_j;\upsilon_j}(c)$.  Furthermore, if $\upsilon_j$ is not a leaf, $t_{s,j}$ stabilizes to some $t_j$ larger than any witness to any lower priority requirement, and there do not exist $b, \vec a$ and $\upsilon$ extending $\upsilon_j$ with $\vec a>t_j$ so that $K_{R_j;\upsilon}((b_0,\ldots,b_{|\upsilon|-1}),\vec a)$, since if there were, we would have taken $\upsilon_{s,j}=\upsilon$ at some stage, so $\Theta_{R_j;\upsilon_j}(c)$ holds.

Finally, we check that $c$ is stable; it suffices to show that for each $n$, there is some $s,i$ such that for all $s'\geq s$, $n\in A_{s,i}$.  But $n$ can only be moved from one $A_i$ to another when some requirement $\leq n$ acts, which only happens finitely many times.
\end{proof}

\section{Separating $\mathbf{SProdWQO}$}

\subsection{Separating from $\mathbf{ADS}$}

In this section we construct a computable instance $c$ of $\mathbf{SProdWQO}$ (and, a fortiori, of $\mathbf{SCAC}$) and a Turing ideal $\mathcal{I}$ which has no solution to $c$, but does satisfy both $\mathbf{trRT^2_k}$ for all $k$ and $\mathbf{WKL}$.

\begin{definition}
  An \emph{\SPROD} is a \ADSSTS{} $R=(T,\{K_\alpha\}_{\sigma\in T},\{d_\sigma\}_{\sigma\in T})$ with range $\{0,1,2\}$ transitive in both colors $1$ and $2$.
\end{definition}

Lemmata \ref{thm:ads_sts_diat}, \ref{thm:ads_sts_wkl_gen} and \ref{thm:ads_sts_exists_gen} apply with $J=\{1,2\}$, $I=\{0,1,2\}$, so we have:
\begin{lemma}
  If $c$ satisfies all \SPROD{}s in $X$ then whenever $B$ is an $X$-computable infinite set, there exist $a,b,c,d\in B$ with $c(a,b)=1$ and $c(c,d)=2$.
\end{lemma}
\begin{lemma}
  If $c$ satisfies all \SPROD{}s in $X$ and $U$ is an infinite $X$-computable $\{0,1\}$-branching tree then there is an infinite branch $\Lambda$ so that $c$ satisfies all \SPROD{}s in $X\oplus\Lambda$.
\end{lemma}
\begin{lemma}
  There is a computable stable $c:[\mathbb{N}]^2\rightarrow\{0,1,2\}$ transitive in the colors $1$ and $2$ satisfying every \SPROD{} in $\emptyset$.
\end{lemma}

We first give our argument showing that we can satisfy $\mathbf{ADS}$.
\begin{lemma}\label{thm:ads_prod_ads}
  Suppose $c$ satisfies every \SPROD{} in $X$ and $\prec$ is a linear ordering.  Then there is an infinite $\prec$-monotone sequence $\Lambda$ so that $c$ satisfies every \SPROD{} in $X\oplus\Lambda$.
\end{lemma}
\begin{proof}
The proof is similar to the proof of Lemma \ref{thm:ads_sts_ads}.  Again, it suffices to assume that $\prec$ is stable-ish as witnessed by $V$, and we again force with conditions $(p,q)$ where $p^+\in V$, $q^+\not\in V$.  Again, it suffices to show:
\begin{quote}
  $(\ast)$ Suppose $R^+$ and $R^-$ are requirements and $(p,q)$ is a condition.  Then there is a condition $(p',q')$ extending $(p,q)$ which either forces $R^+$ on the increasing side or $R^-$ on the decreasing side.
\end{quote}
Let $R^+=(T^+,\{L_\sigma\},\{d^+_\sigma\})$ and $R^-=(T^-,\{M_\tau\},\{d^-_\tau\})$ be given.  As in Lemma \ref{thm:ads_sts_ads}, we can assume that $d^+_\sigma(|\sigma|-1)=0$ for any $\sigma\in T^+$, and a similar assumption for $T^-$.  Recall that a split pair is a pair $(p',q')$ with $p\sqsubseteq p'$, $q\sqsubseteq q'$, and $(p')^+=(q')^+$.  

  The basic idea---combining split pairs of various lengths---is the same as in Lemma \ref{thm:ads_sts_cac}.  However in the proof of Lemma \ref{thm:ads_sts_cac}, we had many split pairs which were all active simultaneously.  To deal with the transitivity requirement, we want to deactivate some split pairs while we are in the process of constructing others.

  In particular, when we obtain a split pair $(p,q)$, we want to ensure that no segment of $p$ (other than the last one) was active at any stage where any segment of $q$ (other than the last one) was constructed and vice-versa.  (Furthermore, because of transitivity, we should assume that if a segment is active at a stage constructing a new segment of any sequence, it is also active at any stage where that new segment is active.)

There is no obstacle in the case where $T^+=T^-=\{\langle\rangle,\langle 0\rangle\}$.   As in the proof of Lemma \ref{thm:ads_sts_cac}, we can have a tree with just two nodes, $\langle\rangle$ and $\langle 0\rangle$, whre $K^X_{\langle 0\rangle}((b_0),\vec a_0)$ holds when $b_0=(e_0,p_0,f_0,q_0)$, $(p_0,q_0)$ is a split pair, $p\sqsubset p_0$, $q\sqsubset q_0$, and both $L^{X\oplus p_0}_{\langle 0\rangle}((e_e),\vec a_0)$ and $M_{\langle 0\rangle}^{X\oplus q_0}((f_0),\vec a_0)$ hold.  Further, notice that $\Theta^X_{R;\langle\rangle}(c)$ will imply either $\Theta^{X\oplus\Lambda}_{R^+;\langle\rangle}(c)$ or $\Theta^{X\oplus\Lambda}_{R^-;\langle\rangle}$ as in Lemma \ref{thm:ads_sts_cac}.

Next, suppose we have $T^+=T^-=\{\langle\rangle, \langle0\rangle,\langle0,0\rangle\}$, and suppose we want to find a split pair $(p,q)$ where $p$ witnesses $\langle 0\rangle$ and $q$ witnesses $\langle 0,0\rangle$.  Then we can arrange to have a tree of four nodes, indicated in Figure \ref{fig:SProd_sub}, which is again essentially identical to the process described in Lemma \ref{thm:ads_sts_cac}.

\begin{figure}
\begin{tikzpicture}[scale=0.6, font=\tiny]
\begin{scope}[shift={(0,0)}]
  \draw(-1.3,-.3)--(1.3,-.3)--(1.3,.8)--(-1.3,.8)--cycle;
  \draw[->](-1,0)--(-.1,.5); \draw[->](1,0)--(.1,.5);
\end{scope}
\begin{scope}[shift={(3,-2)}]
  \draw(-1.3,-.3)--(2.3,-.3)--(2.3,.8)--(-1.3,.8)--cycle;
  \draw[->] (-1,0)--(0,.25); \draw[dashed,->] (0,.25)--(.9,.5);
  \draw[dashed,->](2,0)--(1.1,.5);
\end{scope}
\begin{scope}[shift={(3,-4)}]
  \draw(-1.3,-.3)--(2.3,-.3)--(2.3,.8)--(-1.3,.8)--cycle;
  \draw[dotted,->](-1,0)--(-.1,.5); 
  \draw[dashed,->](2,0)--(1,.25);\draw[dotted,->](1,.25)--(.1,.5);
\end{scope}
\begin{scope}[shift={(-3,-2)}]
  \draw(-1.3,-.3)--(2.3,-.3)--(2.3,.8)--(-1.3,.8)--cycle;
  \draw[dotted,->](-1,0)--(-.1,.5); 
  \draw[->](2,0)--(1,.25);\draw[dotted,->](1,.25)--(.1,.5);
\end{scope}
\draw[->](0,-.3)--(0,-.8)--(3.5,-.8)--(3.5,-1.2);
\draw[->](0,-.8)--(-2.5,-.8)--(-2.5,-1.2);
\draw[->](3.5,-2.3)--(3.5,-3.2);
\end{tikzpicture}
\caption{}\label{fig:SProd_sub}  
\end{figure}

More formally, we have four nodes, $\upsilon$, $\upsilon^\frown\langle 0\rangle$, $\upsilon^\frown\langle 1\rangle$, and $\upsilon^\frown\langle 1,1\rangle$, each $b_i=(e_i,p_i,f_i,q_i)$, and, for instance, $K^X_{\upsilon^\frown\langle 1\rangle}((b_0,\ldots,b_{|\upsilon|}),\vec a_{|\upsilon|})$ holds if $(p_{|\upsilon|},q_{|\upsilon|})$ is a split pair, $p_{|\upsilon|-1}\sqsubseteq p_{|\upsilon|}$, $q\sqsubseteq q_{|\upsilon|}$, $L^{X\oplus p_{|\upsilon|}}_{\langle 0,0\rangle}((e_{|\upsilon|-1},e_{|\upsilon|}),\vec a_{|\upsilon|})$, and $M^{X\oplus q_{|\upsilon|}}_{\langle 0\rangle}((f_{|\upsilon|}),\vec a_{|\upsilon|})$.

Now consider the same case, where $T^+=T^-=\{\langle\rangle, \langle0\rangle,\langle0,0\rangle\}$, but suppose we want to find a split pair $(p,q)$ where both sequences witness $\langle 0,0\rangle$.  

\begin{figure}
\begin{tikzpicture}[scale=0.6, font=\tiny]
  \begin{scope}[shift={(0,0)}]
    \draw[double](-1.3,-.3)--(1.3,-.3)--(1.3,.8)--(-1.3,.8)--cycle;
    \draw[->](-1,0)--(-.1,.5); \draw[->](1,0)--(.1,.5);
    \node at (-1,1) {$\langle 0\rangle$};
    \node at (-.5,0) {$p_0$}; \node at (.5,0) {$q_0$};
  \end{scope}
  \begin{scope}[shift={(0,-2)}]
    \draw(-1.3,-.3)--(1.3,-.3)--(1.3,.8)--(-1.3,.8)--cycle;
    \draw[->](-1,0)--(-.1,.5); \draw[->](1,0)--(.1,.5);
    \node at (-1,1) {$\langle 0,0\rangle$};
    \node at (-.5,0) {$p_1$}; \node at (.5,0) {$q_1$};
  \end{scope}
  \draw[->](0,-.3)--(0,-1.2);
  \begin{scope}[shift={(-2.5,-4)}]
    \draw(-1.3,-.3)--(2.3,-.3)--(2.3,.8)--(-1.3,.8)--cycle;
    \draw[->](-1,0)--(-.1,.5); 
    \draw[->](2,0)--(1,.25);\draw[->](1,.25)--(.1,.5);
    \node at (-.8,1) {$\langle 0,0,0\rangle$};
    \node at (-.5,0) {$p_2$}; \node at (.75,0) {$q_1\sqsubset q_2$};
  \end{scope}
  \begin{scope}[shift={(2.5,-4)}]
    \draw(-1.3,-.3)--(2.3,-.3)--(2.3,.8)--(-1.3,.8)--cycle;
    \draw[->] (-1,0)--(0,.25); \draw[->] (0,.25)--(.9,.5);
    \draw[->](2,0)--(1.1,.5);
    \node at (-.8,1) {$\langle 0,0,1\rangle$};
    \node at (.1,0) {$p_1\sqsubset p_2$}; \node at (1.4,0) {$q_2$};
  \end{scope}
  \draw(0,-2.3)--(0,-2.8);
  \draw(0,-2.8)--(3,-2.8);\draw[->](3,-2.8)--(3,-3.2);
  \draw(0,-2.8)--(-2,-2.8);\draw[->](-2,-2.8)--(-2,-3.2);
  \begin{scope}[shift={(2.5,-6)}]
    \draw(-1.3,-.3)--(2.3,-.3)--(2.3,.8)--(-1.3,.8)--cycle;
    \draw[->](-1,0)--(-.1,.5); 
    \draw[->](2,0)--(1,.25);\draw[->](1,.25)--(.1,.5);
    \node at (-.8,1) {$\langle 0,0,1,0\rangle$};
    \node at (-.5,0) {$p_3$}; \node at (.75,0) {$q_2\sqsubset q_3$};
  \end{scope}
  \draw[->](3,-4.3)--(3,-5.2);
  \begin{scope}[shift={(4.5,-8.5)}]
    \draw(-3.8,-.3)--(2.3,-.3)--(2.3,.8)--(-3.8,.8)--cycle;
    \draw[->](-1,0)--(-.1,.5); 
    \draw[->](2,0)--(1,.25);\draw[->](1,.25)--(.1,.5);
    \draw[->](-3.5,0)--(-2.6,.5); \draw[->](-1.5,0)--(-2.4,.5);
    \node at (-3,0) {$p_0$}; \node at (-2,0) {$q_0$};
    \node at (-.8,1) {$\langle 0,0,0\rangle$ or $\langle 0,0,1,0\rangle$, redrawn};
    \node at (-.5,0) {$p_3$}; \node at (.75,0) {$q_2\sqsubset q_3$};
  \end{scope}
  \begin{scope}[shift={(-4.5,-8.5)}]
    \draw(-3.8,-.3)--(2.3,-.3)--(2.3,.8)--(-3.8,.8)--cycle;
    \draw[->](-3.5,0)--(-2.6,.5); 
    \draw[->](-.5,0)--(-1.5,.25);\draw[->](-1.5,.25)--(-2.4,.5);
    \draw[->](0,0)--(.9,.5); \draw[->](2,0)--(1.1,.5);
    \node at (.5,0) {$p_0$}; \node at (1.5,0) {$q_0$};
    \node at (-.8,1) {$\langle 0,0,0\rangle$ or $\langle 0,0,1,0\rangle$, redrawn};
    \node at (-3,0) {$p_3$}; \node at (-1.75,0) {$q_2\sqsubset q_3$};
  \end{scope}
  \draw[dashed] (-2,-4.3)--(-2,-5.5)--(1.2,-5.5);
  \draw[dashed] (0,-5.5)--(0,-8.2);
  \draw[dashed,->] (0,-8.2)--(.7,-8.2);
  \draw[dashed,->] (0,-8.2)--(-2.2,-8.2);
  \draw[dotted] (-4.2,-.8)--(5.5,-.8)--(5.5,-6.5)--(-4.2,-6.5)--cycle;
  \begin{scope}[shift={(-6,-10.5)}]
    \draw[double](-2.3,-.3)--(2.3,-.3)--(2.3,.8)--(-2.3,.8)--cycle;
    \draw[->] (-2,0)--(-1,.25); \draw[->] (-1,.25)--(-.1,.5);
    \draw[->](2,0)--(1,.25);\draw[->](1,.25)--(.1,.5);
    \node at (-1.2,1) {$\langle \ldots,0\rangle$};
    \node at (-1,0) {$p_3\sqsubset p_4$}; \node at (.75,0) {$q_0\sqsubset q_4$};
  \end{scope}
  \draw(-5.2,-8.8)--(-5.2,-9.2);
  \draw(-5.2,-9.2)--(-6,-9.2);\draw[->](-6,-9.2)--(-6,-9.7);
  \draw[->](3.2,-8.8)--(3.2,-9.7);
  \begin{scope}[shift={(2.5,-10.5)}]
    \draw[double](-1.3,-.3)--(2.3,-.3)--(2.3,.8)--(-1.3,.8)--cycle;
    \draw[->] (-1,0)--(0,.25); \draw[->] (0,.25)--(.9,.5);
    \draw[->](2,0)--(1.1,.5);
    \node at (-.8,1) {$\langle \ldots,1\rangle$};
    \node at (.1,0) {$p_0\sqsubset p_4$}; \node at (1.4,0) {$q_4$};
  \end{scope}
  \draw(-5.2,-9.2)--(3,-9.2);\draw[->](3,-9.2)--(3,-9.7);
  \begin{scope}[shift={(2.5,-10.5)}]
  \begin{scope}[shift={(0,-2)}]
    \draw(-1.3,-.3)--(1.3,-.3)--(1.3,.8)--(-1.3,.8)--cycle;
    \draw[->](-1,0)--(-.1,.5); \draw[->](1,0)--(.1,.5);
    \node at (-1.4,1) {$\langle\ldots,1,0\rangle$};
    \node at (-.5,0) {$p_5$}; \node at (.5,0) {$q_5$};
  \end{scope}
  \draw[->](0,-.3)--(0,-1.2);
  \begin{scope}[shift={(-2.5,-4)}]
    \draw(-1.3,-.3)--(2.3,-.3)--(2.3,.8)--(-1.3,.8)--cycle;
    \draw[->](-1,0)--(-.1,.5); 
    \draw[->](2,0)--(1,.25);\draw[->](1,.25)--(.1,.5);
    \node at (-.8,1) {$\langle \ldots,1,0,0\rangle$};
    \node at (-.5,0) {$p_6$}; \node at (.75,0) {$q_5\sqsubset q_6$};
  \end{scope}
  \begin{scope}[shift={(2.5,-4)}]
    \draw(-1.3,-.3)--(2.3,-.3)--(2.3,.8)--(-1.3,.8)--cycle;
    \draw[->] (-1,0)--(0,.25); \draw[->] (0,.25)--(.9,.5);
    \draw[->](2,0)--(1.1,.5);
    \node at (-.8,1) {$\langle \ldots,1,0,1\rangle$};
    \node at (.1,0) {$p_5\sqsubset p_6$}; \node at (1.4,0) {$q_6$};
  \end{scope}
  \draw(0,-2.3)--(0,-2.8);
  \draw(0,-2.8)--(3,-2.8);\draw[->](3,-2.8)--(3,-3.2);
  \draw(0,-2.8)--(-2,-2.8);\draw[->](-2,-2.8)--(-2,-3.2);
  \begin{scope}[shift={(2.5,-6)}]
    \draw(-1.3,-.3)--(2.3,-.3)--(2.3,.8)--(-1.3,.8)--cycle;
    \draw[->](-1,0)--(-.1,.5); 
    \draw[->](2,0)--(1,.25);\draw[->](1,.25)--(.1,.5);
    \node at (-1.2,1) {$\langle\ldots,1,0,1,0\rangle$};
    \node at (-.5,0) {$p_7$}; \node at (.75,0) {$q_6\sqsubset q_7$};
  \end{scope}
  \draw[->](3,-4.3)--(3,-5.2);
  \begin{scope}[shift={(5.5,-8.5)}]
    \draw(-4.8,-.3)--(2.3,-.3)--(2.3,.8)--(-4.8,.8)--cycle;
    \draw[->](-1,0)--(-.1,.5); 
    \draw[->](2,0)--(1,.25);\draw[->](1,.25)--(.1,.5);
    \draw[->](-1.5,0)--(-2.4,.5);
    \draw[->] (-4.5,0)--(-3.5,.25); \draw[->] (-3.5,.25)--(-2.6,.5);
    \node at (-3.4,0) {$p_0\sqsubset p_4$}; \node at (-2,0) {$q_4$};
    \node at (-.8,1) {$\langle \ldots,1,0,0\rangle$ or $\langle\ldots,1,0,1,0\rangle$ redrawn};
    \node at (-.5,0) {$p_7$}; \node at (.75,0) {$q_6\sqsubset q_7$};
  \end{scope}
  \begin{scope}[shift={(-5.5,-8.5)}]
    \draw(-3.8,-.3)--(3.3,-.3)--(3.3,.8)--(-3.8,.8)--cycle;
    \draw[->](-3.5,0)--(-2.6,.5); 
    \draw[->](-.5,0)--(-1.5,.25);\draw[->](-1.5,.25)--(-2.4,.5);
    \draw[->] (0,0)--(1,.25); \draw[->] (1,.25)--(1.9,.5);
    \draw[->](3,0)--(2.1,.5);
    \node at (.8,0) {$p_0\sqsubset p_4$}; \node at (2.5,0) {$q_4$};
    \node at (-.8,1) {$\langle \ldots,1,0,0\rangle$ or $\langle\ldots,1,0,1,0\rangle$ redrawn};
    \node at (-3,0) {$p_7$}; \node at (-1.75,0) {$q_6\sqsubset q_7$};
  \end{scope}
  \draw[dashed] (-2,-4.3)--(-2,-5.5)--(1.2,-5.5);
  \draw[dashed] (0,-5.5)--(0,-8.2);
  \draw[dashed,->] (0,-8.2)--(.7,-8.2);
  \draw[dashed,->] (0,-8.2)--(-2.2,-8.2);
  \draw[dotted] (-4.6,-.8)--(5.5,-.8)--(5.5,-6.5)--(-4.6,-6.5)--cycle;
\end{scope}
\begin{scope}[shift={(-3.25,-21)}]
  \draw[double](-2.3,-.3)--(2.3,-.3)--(2.3,.8)--(-2.3,.8)--cycle;
  \draw[->] (-2,0)--(-1,.25); \draw[->] (-1,.25)--(-.1,.5);
  \draw[->](2,0)--(1,.25);\draw[->](1,.25)--(.1,.5);
  \node at (-1.2,1) {$\langle \ldots,\ldots,0\rangle$};
  \node at (-1,0) {$p_7\sqsubset p_8$}; \node at (.75,0) {$q_4\sqsubset q_8$};
\end{scope}
\draw[->] (-3.25,-19.3)--(-3.25,-20.2);
\end{tikzpicture}
\caption{}\label{fig:SProd_iteration}  
\end{figure}

We illustrate the process in Figure \ref{fig:SProd_iteration}, and will now go through the steps to clarify the diagram.  As drawn, there are some redundancies and inefficiencies, but these reflect how our actual construction will be built recursively.

First, we explain the notion used in the diagram and the underlying tree it represents, and then explain how it is obtained.  Each $b_i=(e_i,p_i,f_i,q_i)$, where $p_i$ and $q_i$ are a split pair and $e_i$ and $f_i$ are the auxiliary date for $L$ and $M$.  Each box labeled $\upsilon$ indicates the configuration that is promised to exist by $K^X_\upsilon$.  For example, $K^X_{\langle 0,0,1,0\rangle}((b_0,\ldots,b_3),\vec a_3)$ holds when $(p_4,q_4)$ is a split pair, $q_3\sqsubset q_4$, $L^{X\oplus p_4}_{\langle 0\rangle}((e_4),\vec a_3)$, and $M^{X\oplus q_r}_{\langle 0,0\rangle}((f_3,f_4),\vec a_3)$.

The nodes $\langle 0,0,0\rangle$ and $\langle 0,0,1,0\rangle$ each have the same subtree below them, so we only copy it once.  For instance, the node $\langle \ldots,0\rangle$ refers to two nodes---$\langle 0,0,0,0\rangle$ and $\langle 0,0,1,0,0\rangle$---which are largely identical.  (However the indices come from descendents of the longer branch; for example, the node $\langle 0,0,0,0\rangle$ should actually be labeled $p_3\sqsubset p_4, q_1\sqsubset q_4$.)  This occurs again at the very end, where the nodes $\langle \ldots,1,0,0\rangle$ and $\langle \ldots,1,0,1,0\rangle$ have the same subtree (consisting of a single node) below them.

The definition of $d$ is that segments are inactive when a segment is outside a dotted box and the child nodes are inside the box.  For example, $d_{\langle \ldots,1\rangle}(0)=0$, because the segments $(p_0,q_0)$ were constructed outside the box and the child of $\langle \ldots,0\rangle$ is inside the box.  However $d_{\langle \ldots,1,0,0\rangle}(0)=d^+_{\langle 0,0\rangle}(0)$ because the split pair $(p_5,q_5)$ should be active in the construction of children of $\langle \ldots,1,0,0\rangle$.

The boxes with two split pairs are comparing the order of the endpoints---in the left copy of ``$\langle 0,0,0\rangle$ or $\langle 0,0,1,0\rangle$, redrawn'', the node $p_3^+\prec p_0^+$ while in the right copy, $p_0^+\prec p_3^+$.  These two situations can lead to slightly different possible outcomes, so we illustrate them separately.

As an example, we go through our analysis when $\Theta^X_{T;\langle 0,0,1,0\rangle}(c)$ holds.  Since $\Delta^X_{T;\langle 0,0,1,0\rangle}(c,(b_0,\ldots,b_3),\vec a_0,\ldots,\vec a_3)$ holds, we have the split pairs $(p_0,q_0)$ and $(p_3,q_3)$ where $L^{X\oplus p_3}_{\langle 0\rangle}((e_3),\vec a_3)$, $M^{X\oplus q_3}_{\langle 0,0\rangle}((f_2,f_3),\vec a_3)$, and so on.  If $p_3^+\prec p_0^+$---that is, the left hand case---then either there are infinitely many $x$ with $x\prec p_3^+$, infinitely many $x$ with $p_3^+\prec x\prec p_0^+$, or infinitely many $x$ with $p_0^+\prec x$.  If there are infinitely many $x$ with $x\prec p_3^+$ then $\Theta^{X\oplus q_3}_{T^-;\langle 0,0\rangle}(c)$ holds.  If there are infinitely many $x$ with $p_3^+\prec x\prec p_0^+$ then one of $\Theta^{X\oplus p_3}_{T^+;\langle 0\rangle}(c)$ and $\Theta^{X\oplus q_0}_{T^-;\langle 0\rangle}(c)$ must hold (because otherwise we would be able to find witnesses to the node $\langle \ldots,0\rangle$).  If there are infinitely many $x$ with $p_0^+\prec x$ then $\Theta^{X\oplus p_3}_{T^+;\langle 0\rangle}(c)$ must hold (because otherwise we would be able to find witnesses to the node $\langle \ldots,1\rangle$).

In the right hand case, where $p_0^+\prec p_3^+$, the situation is simpler: we only care about whether there are infinitely many $x$ with $x\prec p_3^+$ or infinitely many $x$ with $p_3^+\prec x$.  If there are infinitely many $x$ with $x\prec p_3^+$ then we have $\Theta^{X\oplus q_3}_{T^-;\langle 0,0\rangle}(c)$.  If there are infinitely many $x$ with $p_3^+\prec x$ then also there are infinitely many $x$ with $p_0^+\prec x$, so we have $\Theta^{X\oplus p_0}_{T^+;\langle 0\rangle}(c)$ (since otherwise we would find witnesses to $\langle \ldots,1\rangle$).

Similar analyses (usually with fewer cases) hold at other nodes.

We now point out how this tree is built.  The way our recursion works is that we will build constructions of longer split pairs by combining the trees that build short ones.  In particular, we will take a ``sub-process''---that is a tree of nodes producing some particular configuration---and insert it into a second tree (the ``main process'').  In Figure \ref{fig:SProd_iteration}, the four nodes in the dotted box represent the sub-process, which in this case is the tree from Figure \ref{fig:SProd_sub}, which is repeated twice.  At the end node of each of these subprocesses, we have ensured the construction of a split pair $(p',q')$ where $p'$ has one segment and $q'$ has two segments.

In this case the main process is actually the same process: the four nodes with doubled borders actually form the same underlying tree.  We produce this by beginning with the four nodes from the original process, identical to those in the dotted box.  However every time were are at a node one of whose children is a leaf, we insert a copy of our subprocess.

Consider the first time this happens.  The nodes $\langle 0\rangle$ and $\langle\ldots,0\rangle$ in the larger tree correspond to the nodes $\langle 0,0\rangle$ and $\langle 0,0,0\rangle$ in the dotted box.  In the sub-process, this corresponds to extending a decreasing segment of length $1$ to a decreasing segment of length $2$, paired with a new segment of length $1$.  In the passage from $\langle 0\rangle$ to $\langle \ldots,0\rangle$, however, we pair this segment with an increasing sequence of length $2$---we use the sub-process to obtain a second, unrelated, split pair, and we use $p_3$ from that pair as the basis for forming a longer increasing sequence.

Note that this tree is simpler than a general tree for constructing split pairs where both segments have length $2$, because we are taking advantage of the fact that we never build segments of length greater than $2$.  In general, there would have to be additional side branches corresponding to cases where, instead, one of our segments of length $2$ was extended to a segment of length $3$.

We now describe our general construction.  Let $r=\max\{|\sigma|\mid\sigma\in T^+\}$ and $s=\max\{|\tau|\mid \tau\in T^-\}$.  Let $D$ be the set of pairs $(r',s')$ with $r'\in[1,r]$ and $s'\in [1,s]$.

When $D'$ is a set of pairs, a \emph{process of type $D'$} is a requirement $R_{D'}=(T_{D'},\{K_{D',\upsilon}\},\{d_{K',\upsilon}\})$ such that each leaf constructs a split pair whose lengths belong to $D'$.  Stated formally, for each leaf $\upsilon\in T_{D'}$, $\Delta^X_{D',\upsilon}(c,(b_0,\ldots,b_{|\upsilon|-1}),\vec a_0,\ldots,\vec a_{|\upsilon|-1})$ implies that each $b_i$ has the form $(e_i,p_i,f_i,q_i)$ where:
\begin{itemize}
\item $p\sqsubseteq p_i$,
\item $q\sqsubseteq q_i$,
\item $(p_i,q_i)$ is a split pair,
\item there are $\sigma\in T^+$ and $\tau\in T^-$ and sequences $(v_0,\ldots,v_{|\sigma|-1})$ and $(w_0,\ldots,w_{|\tau|-1})$ such that:
  \begin{itemize}
  \item the sequences are disjoint except that $v_{|\sigma|-1}=w_{|\tau|-1}$,
  \item $(|\sigma|,|\tau|)\in D'$,
  \item $\Delta^{X\oplus p_i}_{T^+;\sigma}(c,(e_{v_0},\ldots,e_{v_{|\sigma|-1}}),\vec a_{v_0},\ldots,\vec a_{v_{|\sigma|-1}})$,
  \item $\Delta^{X\oplus q_i}_{T^-;\tau}(c,(f_{w_0},\ldots,f_{w_{|\sigma|-1}}),\vec a_{w_0},\ldots,\vec a''_{w_{|\tau|-1}})$.
  \end{itemize}
\end{itemize}
Furthermore, we require that for each non-leaf $\upsilon\in T_{D'}$, $\Theta_{R_{D'};\upsilon}^X(c)$ implies that either there is a $p'$ forcing $R^+$ on the chain side or a $q'$ forcing $R^-$ on the antichain side.  (For notational reasons, we allow $D'\not\subseteq D$, however note that a process of type $D'$ is equivalent to a process of type $D'\cap D$.)

Our main construction will show that, given a process of type $\{(1,s')\}$ and a process of type $\{(1,s'+1),(r',s')\}$, we can produce a process of type $\{(1,s'+1),(r'+1,s')\}$.

We have constructed a process of type $\{(1,1)\}$: for each $\sigma\in T^+$ and $\tau\in T^-$, $T_{\{(1,1)\}}$ has a node $\langle (\sigma,\tau)\rangle$, with $\Delta^{X}_{R_{\{1,1\}};\langle (\sigma,\tau)\rangle}(c,(e_0,p_0,q_0,f_0),\vec a_0)$ holding when both $\Delta^{X\oplus p_0}_{R^+;\sigma}(c,(e_0),\vec a_0)$ and $\Delta^{X\oplus q_0}_{R^-;\tau}(c,(f_0),\vec a_0)$ hold (and also the usual conditions---$(p_0,q_0)$ are a split pair with $p\sqsubset p_0$ and $q\sqsubset q_0$).

For the recursive part of the construction, suppose we have $R_{\{(1,s'+1),(r',s')\}}$ and $R_{\{(1,s')\}}$.  We describe $R_{\{(1,s'+1),(r'+1,s')\}}$.  Roughly speaking, we will copy $R_{\{(1,s')\}}$ except that, before each leaf, we will insert a copy of $R_{\{(1,s'+1),(r',s)\}}$ and modify the leaf accordingly.

We construct $T_{\{(1,s'+1),(r'+1,s')\}}$ and, as we do, a partial function $f:T_{\{(1,s'+1),(r'+1,s')\}}\rightarrow T_{\{(1,s')\}}$ and, for each $\upsilon\in \dom(f)$, a monotone reindexing function $\pi_\upsilon:[0,|f(\upsilon)|)\rightarrow[0,|\upsilon|)$.  We set $f(\langle\rangle)=\langle\rangle$.  Suppose $\upsilon\in\dom(f)$.

If no children of $f(\upsilon)$ are leaves of $T_{\{(1,s')\}}$ then we copy the children of $f(\upsilon)$ to be the children of $\upsilon$: for each child $f(\upsilon)^\frown\langle x\rangle\in T_{\{(1,s')\}}$, we place a node $\upsilon^\frown\langle x\rangle\in T_{\{(1,s'+1),(r'+1,s')\}}$ with $f(\upsilon^\frown\langle x\rangle)=f(\upsilon)^\frown\langle x\rangle$, and set:
\begin{itemize}
\item $\pi_{\upsilon^\frown\langle x\rangle}=\pi_\upsilon\cup\{(|f(\upsilon)|,|\upsilon|)\}$,
\item for each $i<|f(\upsilon)|$, $d_{\{(1,s'+1),(r'+1,s)\},\upsilon}(\pi_{\upsilon}(i))=d_{\{(1,s')\},f(\upsilon)}(i)$,
\item $K^X_{\{(1,s'+1),(r'+1,s)\},\upsilon^\frown\langle x\rangle}((b_0,\ldots,b_{|\upsilon|}),\vec a_{|\upsilon|})$ holds exactly when 
\[K^X_{\{(1,s'+1)\},f(\upsilon)^\frown\langle x\rangle}((b_{\pi_{\upsilon^\frown\langle x\rangle}(0)},\ldots,b_{\pi_{\upsilon^\frown\langle x\rangle}(|f(\upsilon)|)}),\vec a_{\pi_{\upsilon^\frown\langle x\rangle}(|f(\upsilon)|)}).\]
\end{itemize}

Suppose that a child of $f(\upsilon)$ is a leaf of $T_{\{(1,s')\}}$.  Then we first place a copy of $T_{\{(1,s'+1),(r'+1,s')\}}$: for each $\zeta\in T_{\{(1,s'+1),(r',s')\}}$, we have a node $\upsilon^\frown\zeta\in T_{\{(1,s'+1),(r'+1,s')\}}$ with:
\begin{itemize}
\item for each $i<|\upsilon|$, $d_{\{(1,s'+1),(r'+1,s')\},\upsilon^\frown\zeta}(i)=0$,
\item for each $i\in[|\upsilon|,|\upsilon^\frown\zeta|)$, $d_{\{(1,s'+1),(r'+1,s')\},\upsilon^\frown\zeta}(i)=d_{\{(1,s'+1),(r',s')\},\zeta}(i-|\upsilon|)$,
\item for $\zeta\neq\langle\rangle$, $K^X_{\{(1,s'+1),(r'+1,s')\},\upsilon^\frown\zeta}((b_0,\ldots,b_{|\upsilon|+|\zeta|-1}),\vec a_{|\upsilon|+|\zeta|-1})$ holds exactly when $K^X_{\{(1,s'+1),(r',s')\},\zeta}((b_{|\upsilon|},\ldots,b_{|\upsilon|+|\zeta|-1}),\vec a_{|\upsilon|+|\zeta|-1})$ holds.
\end{itemize}

Consider a leaf $\zeta\in T_{\{(1,s'+1),(r',s')\}}$.  There is some pair $(\sigma,\tau)$ with $(|\sigma|,|\tau|)\in\{(1,s'+1),(r',s')\}$ corresponding to this leaf.  If $(|\sigma|,|\tau|)=(1,s'+1)$ then $\upsilon^\frown\zeta$ is a leaf of $T_{\{(1,s'+1),(r'+1,s')\}}$ as well.

So suppose $(|\sigma|,|\tau|)=(r',s')$.  We have corresponding indices $v_0,\ldots,v_{|\sigma|-1}$ and $w_0,\ldots,w_{|\tau|-1}$.   We set $d_{\{(1,s'+1),(r'+1,s')\},\upsilon^\frown\zeta}$ by:
\begin{itemize}
\item for $i<|f(\upsilon)|$, $d_{\{(1,s'+1),(r'+1,s')\},\upsilon^\frown\zeta}(\pi_\upsilon(i))=d_{\{(1,s'+1)\}}(i)$,
\item for $i<|\sigma|$, $d_{\{(1,s'+1),(r'+1,s')\},\upsilon^\frown\zeta}(|\upsilon|+v_i)=d^+_{\sigma}(i)$,
\item for $i<|\tau|$, $d_{\{(1,s'+1),(r'+1,s')\},\upsilon^\frown\zeta}(|\upsilon|+w_i)=d^-_{\tau}(i)$.
\end{itemize}
Note that we have arranged for this choice of $d_{\{(1,s'+1),(r'+1,s')\},\upsilon^\frown\zeta}$ to be consistent with the transitivity requirements.


For each $f(\upsilon)^\frown\langle x\rangle\in T_{\{(1,s')\}}$ which is not a leaf, we have a node $\upsilon^\frown\zeta^\frown\langle x\rangle\in T_{\{(1,s'+1),(r'+1,s')\}}$ with $f(\upsilon^\frown\zeta^\frown\langle x\rangle)=f(\upsilon)^\frown\langle x\rangle$ and $\pi_{\upsilon^\frown\zeta^\frown\langle x\rangle}=\pi_\upsilon\cup\{(|\upsilon|,|\upsilon|+|\zeta|)\}$.  In this case, $d_{\{(1,s'+1),(r'+1,s')\},\upsilon^\frown\zeta^\frown\langle x\rangle}$ and $K^X_{\{(1,s'+1),(r'+1,s')\},\upsilon^\frown\zeta^\frown\langle x\rangle}$ are copied from $T_{\{(1,s')\}}$ just like the case where no child of $f(\upsilon)$ was a leaf.

Consider a leaf $f(\upsilon)^\frown\langle x\rangle\in T_{\{(1,s')\}}$; it must be associated to a pair $(\sigma',\tau')$ with $|\sigma'|=1$ and $|\tau'|=s'$.  We have a corresponding sequence of witnesses $w'_0,\ldots,w'_{|\tau'|-1}$.  For each immediate extension $\sigma''$ of $\sigma$, have a leaf $\upsilon^\frown\zeta^\frown\langle (1,\sigma',x)\rangle$, and $K^X_{\{(1,s'+1),(r'+1,s')\},\upsilon^\frown\zeta^\frown\langle (1,\sigma',x)\rangle}((b_0,\ldots,b_{|\upsilon|+|\zeta|}),\vec a_{|\upsilon|+|\zeta|})$ to hold when 
\begin{itemize}
\item $(p_{|\upsilon|+|\zeta|},q_{|\upsilon|+|\zeta|})$ is a split pair,
\item $p_{|\upsilon|+|\zeta|}$ extends $p_{|\upsilon|+|\zeta|-1}$,
\item $q_{|\upsilon|+|\zeta|-1}$ extends $q_{\pi_\upsilon(w'_{|\tau'|-2})}$,
\item $L^{X\oplus p_{|\upsilon|+|\zeta|}}_{T^+;\sigma''}((e_{|\upsilon|+v_0},\ldots,e_{|\upsilon|+v_{|\sigma|-1}},e_{|\upsilon|+|\zeta|-1}),\vec a_{|\upsilon|+|\zeta|})$ holds,
\item $M^{X\oplus q_{|\upsilon|+|\zeta|}}_{T^-;\tau'}((f_{\pi_\upsilon(0)},\ldots,f_{\pi_\upsilon(|\tau'|-1)},f_{|\upsilon|+|\zeta|}),\vec a_{|\upsilon|+|\zeta|})$ holds.
\end{itemize}

Iteration of this method gives the desired process.  We have a process of type $\{(1,1)\}$.  Given a process of type $\{(1,s')\}$, we apply this combination to obtain a process of type $\{(1,s'+1),(2,s')\}$, and by repeating $\{(1,s'+1),(r',s')\}$ for any $r'$.  In particular, we get a process of type $\{(1,s'+1),(r+1,s')\}$, which is the same as a process of type $\{(1,s'+1)\}$.  Inductively, we have processes of type $\{(1,s')\}$ for all $s'$.  In particular, applying the first iteration again, we have processes of type $\{(1,s+1),(r',s)\}$ for each $r'$, which is the same as a process of type $\{(r',s)\}$.  Finally, we obtain a process of type $\{(r,s)\}$, which suffices to give the desired extensions.
\end{proof}

\subsection{Separating from $\mathbf{trRT^2_k}$}

We need to generalize the ideas of the previous subsection to $\mathbf{trRT^2_k}$.  The general ideas are the same, but the bookkeeping is slightly more complicated because we now have $k$ different processes we need to interleave.

\begin{lemma}\label{thm:ads_prod_trrt}
  Suppose $c$ satisfies every \SPROD{} in $X$ and $c^*:[\mathbb{N}]^2\rightarrow[1,k]$ with all colors transitive.  Then there is an infinite $c^*$-homogeneous set $S$ so that $c$ satisfies every \SPROD{} in $X\oplus\Lambda$.
\end{lemma}
\begin{proof}
Our conditions are tuples $(p_1,\ldots,p_k)$ where each $p_i$ is homogeneously colored $i$ and there are infinitely many $x$ so that, for each $a\in p_i$, $c^*(a,x)=i$.  Given requirements $R_1,\ldots,R_k$, we must find a condition $(p'_1,\ldots,p'_k)$ with each $p_i\sqsubseteq p'_i$ so that some $R_i$ is forced.

  A \emph{split $k$-tuple} is a tuple $(q_1,\ldots,q_k)$ with each $p_i\sqsubseteq q_i$ and $(q_1)^+=\cdots=(q_k)^+$; it follows that there is at least one $i_0$ so that, taking $p'_{i_0}=q_{i_0}$ and $p'_i=p_i$ for $i\neq i_0$, $(p'_1,\ldots,p'_k)$ is a condition.

  For each $i\leq k$, $r_i=\max\{|\sigma|\mid \sigma\in T_i\}$, and we take $D=\prod_i [1,r_i]$.  The notion of constructing a split tuple of type $(r'_1,\ldots,r'_k)\in D$ and a process of type $D'\subseteq D$ are given by the generalizations of the corresponding notions from the previous subsection.

We can describe a process of type $\{(1,\ldots,1)\}$: for each sequence $(\sigma_1,\ldots,\sigma_k)$ with $\sigma_i\in T_i$ and $|\sigma_i|=1$, we have a node $\langle (\sigma_1,\ldots,\sigma_k)\rangle$ where $K^X_{R_{\{(1,\ldots,1)\}};\langle (\sigma_1,\ldots,\sigma_k)\rangle}((b_0),\vec a_0)$ holds when 
\begin{itemize}
\item $b_0=(e_0^0,p_0^0,\ldots,e_0^k,p_0^k)$,
\item $(p_0^0,\ldots,p_0^k)$ is a split $k$-tuple,
\item $p_i\sqsubset p_0^i$ for each $i\leq k$,
\item $K^{X\oplus p_i}_{R_i;\sigma_i}((e_0^i),\vec a_0)$ for each $i\leq k$.
\end{itemize}

We want to work towards processes of ``larger'' type.  It is clear that, say, finding a split tuple of type $(1,2,2)$ represents more progress than a tuple of type $(1,2,1)$; we work lexicographically, so we also consider a tuple $(1,1,2)$ to be further progress than a tuple of type $(1,3,1)$.  (This is consistent with what we did above, where we considered a slightly longer antichain to be more progress than a much longer chain.)

We place tuples $\vec d$ in reverse lexicographic order, so $\vec d<\vec d'$ if there is an $i$ so that $d_j=d'_j$ for $i<j$, and $d_i<d'_i$.  $(1,\ldots,1)$ is the smallest element in this ordering.  Given some $\vec d\in (d_1,\ldots,d_k)$, we define $\vec d^{+i}=(1,\ldots,1,d_i+1,d_{i+1},\ldots,d_k)$---that is,
\[d^{+i}_j=\left\{\begin{array}{ll}
1&\text{if }j<i\\
d_j+1&\text{if }j=i\\
d_j&\text{if }j>i
\end{array}\right..\]
We define
\[D_{\vec d}=D\cap(\{\vec d\}\cup\{\vec d^{+i}\mid \exists j< i\ d_j\neq 1\}).\]
So $D_{(1,\ldots,1)}=\{(1,\ldots,1)\}$ while 
\[D_{(1,2,1,2,1)}=\{(1,2,1,2,1),(1,1,2,2,1),(1,1,1,3,1),(1,1,1,1,2)\}.\]
We will show by induction on $\vec d$ that we can construct a process of type $D_{\vec d}$.

The basic idea is the same as in the previous subsection: when we want to construct a process of type $(1,1,\ldots,c_i,c_{i+1},\ldots,c_k)$ where $c_i>1$, we take a proces of type $(1,1,\ldots,1,c_{i+1},\ldots,c_k)$; before each step which might be a leaf, we decativate all segments constructed so far and insert a sub-process of type $(r_1,\ldots,r_{i-1},c_i-1,c_{i+1},\ldots,c_k)$.  (Both these processes preceed $(1,1,\ldots,c_i,c_{i+1},\ldots,c_k)$ in our ordering, so we may assume they exist.)  Then we return to the original process, except that we look for the $i$-th sequence in our new tuple to extend the $i$-th sequence created by the inserted sub-process.

Suppose we have constructed a process of type $D_{\vec e}$ for all $\vec e<\vec d$.  Then $\vec d=(c_1,\ldots,c_k)$; since we covered the case of a process of type $\{(1,\ldots,1)\}$ above, we may assume there is some $i$ with $c_i\neq 1$.  Fix $i$ least so that $c_i\neq 1$.

Let $\vec d^-=(1,\ldots,1,c_i-1,c_{i+1},\ldots,c_k)$ and $\vec d_0=(1,\ldots,1,1,c_{i+1},\ldots,c_k)$.  We will obtain our process of type $D_{\vec d}$ as a suitable modification of our process of type $D_{\vec d_0}$.

We begin by copying $T_{D_{\vec d_0}}$: as we construct $T_{D_{\vec d}}$, we define a partial function $f:T_{D_{\vec d}}\rightarrow T_{D_{\vec d_0}}$ and, for each $\upsilon\in \dom(f)$, a monotone function $\pi_\upsilon:[0,|f(\upsilon)|)\rightarrow[0,|\upsilon)$.  We set $f(\langle\rangle)=\langle\rangle$.  Consider some $\upsilon\in\dom(f)$.

If no children of $f(\upsilon)$ are leaves of $D_{\vec d_0}$ which gives a tuple of type $(1,\ldots,1,1,c_{i+1},\ldots,c_k)$ then we copy the children of $f(\upsilon)$: for each child $f(\upsilon)^\frown\langle x\rangle\in T_{D_{\vec d_0}}$, we place a node $\upsilon^\frown\langle x\rangle\in T_{\vec d}$ with $f(\upsilon^\frown\langle x\rangle)=f(\upsilon)^\frown\langle x\rangle$ and we set:
\begin{itemize}
\item $\pi_{\upsilon^\frown\langle x\rangle}=\pi_\upsilon\cup\{(|f(\upsilon)|,|\upsilon|)\}$,
\item for each $i<|f(\upsilon)|$, $d_{D_{\vec d},\upsilon}(\pi_{\upsilon}(i))=d_{D_{\vec d_0},f(\upsilon)}(i)$,
\item $K^X_{D_{\vec d},\upsilon^\frown\langle x\rangle}((b_0,\ldots,b_{|\upsilon|}),\vec a_{|\upsilon|})$ holds exactly when 
\[K^X_{D_{\vec d_0},f(\upsilon)^\frown\langle x\rangle}((b_{\pi_{\upsilon^\frown\langle x\rangle}(0)},\ldots,b_{\pi_{\upsilon^\frown\langle x\rangle}(|f(\upsilon)|)}),\vec a_{\pi_{\upsilon^\frown\langle x\rangle}(|f(\upsilon)|)})\]
 holds.
\end{itemize}

Suppose that some child of $f(\upsilon)$ is a leaf of $T_{D_{\vec d_0}}$ which gives a tuple of type $(1,\ldots,1,1,c_{i+1},\ldots,c_k)$.  Then we place a copy of $T_{D_{\vec d'}}$ below $\upsilon$: for each $\zeta\in T_{D_{\vec d'}}$ we have a node $\upsilon^\frown\zeta\in  T_{D_{\vec d}}$ with:
\begin{itemize}
\item for each $i<|\upsilon|$, $d_{D_{\vec d},\upsilon^\frown\zeta}(i)=0$,
\item for each $i\in [|\upsilon|,\upsilon^\frown\zeta)$, $d_{D_{\vec d},\upsilon^\frown\zeta}(i)=d_{D_{\vec d'},\zeta}(i-|\upsilon|)$,
\item for $\zeta\neq\langle\rangle$, $K^X_{D_{\vec d},\upsilon^\frown\zeta}((b_0,\ldots,b_{|\upsilon|+|\zeta|-1}),\vec a_{|\upsilon|+|\zeta|-1})$ holds exactly when $K^X_{D_{\vec d'},\zeta}((b_{|\upsilon|},\ldots,b_{|\upsilon|+|\zeta|-1}),\vec a_{|\upsilon|+|\zeta|-1})$ holds.
\end{itemize}

Consider a leaf $\zeta$ of $T_{D_{\vec d'}}$, which is witnessed by some tuple $(c'_1,\ldots,c'_k)\in D_{\vec d'}$.  If $(c'_1,\ldots,c'_k)\neq \vec d'$ then also $(c'_1,\ldots,c'_k)\in D_{\vec d}$, so $\zeta$ is a leaf of $T_{D_{\vec d}}$.

So consider a leaf $\zeta$ of $T_{D_{\vec d'}}$ witnessed by $(c'_1,\ldots,c'_k)=\vec d'$, with a corresponding tuple $(\sigma_1,\ldots,\sigma_k)$ with $|\sigma_j|=c'_j$.  Then for each $j<k$ we have a sequence of indices of segments $v^j_0,\ldots,v^j_{c'_j-1}$.  Then
\begin{itemize}
\item for $i<|f(\upsilon)|$, $d_{D_{\vec d},\upsilon^\zeta}(\pi_\upsilon(i))=d_{\vec d_0}(i)$,
\item for $j<k$ and $i<c'_j$, $d_{D_{\vec d},\upsilon^\zeta}(|\upsilon|+v^j_i)=d_{R_j;\sigma_j}(i)$.
\end{itemize}

For each $f(\upsilon)^\frown\langle x\rangle\in T_{D_{\vec d_0}}$ which is not a leaf producing a tuple of type $(1,\ldots,1,1,c_{i+1},\ldots,c_k)$, we have a node $\upsilon^\frown\zeta^\frown\langle x\rangle\in T_{D_{\vec d}}$ with $f(\upsilon^\frown\zeta^\frown\langle x\rangle)=\upsilon^\frown\langle x\rangle$ as in the case above.

Consider some $f(\upsilon)^\frown\langle x\rangle\in T_{D_{\vec d_0}}$ which is a leaf producing a tuple of type $(1,\ldots,1,1,c_{i+1},\ldots,c_k)$ witnessing the nodes $(\tau_1,\ldots,\tau_k)$, where, for $j\geq i+1$, the resulting sequences will come from the segments $w^j_0,\ldots,w^j_{c_j-1}\leq|f(\upsilon)|$.  Then for each $\sigma^*\in T_{R_i}$ an immediate extension of $\sigma_i$, we have a node $\theta=\upsilon^\frown\zeta^\frown\langle (\tau_1,\ldots,\tau_{i-1},\sigma^*,\tau_{i+1},\ldots,\tau_k)\rangle$ where $K^X_{D_{\vec d},\theta}((b_0,\ldots,b_{|\upsilon|+|\zeta|}),\vec a_{|\upsilon|+|\zeta|})$ holds exactly when:
\begin{itemize}
\item $(p^0_{|\upsilon|+|\zeta|},\ldots,p^k_{|\upsilon|+|\zeta|})$ is a split tuple,
\item for $j<i$, $p^j_{|\upsilon|+|\zeta|}$ extends $p^j$,
\item $p^i_{|\upsilon|+|\zeta|}$ extends $p^i_{|\upsilon|+|\zeta|-1}$,
\item for $j>i$, $p^j_{|\upsilon|+|\zeta|}$ extends $p^j_{w_{c_j-1}}$,
\item for $j<i$, $K^{X\oplus p^j_{|\upsilon|+|\zeta|}}_{R_j;\sigma_j}((e^j_{|\upsilon|+|\zeta|}),\vec a_{|\upsilon|+|\zeta|})$ holds,
\item $K^{X\oplus p^i_{|\upsilon|+|\zeta|}}_{R_i;\sigma^*}((e^i_{|\upsilon|+v^i_0},\ldots,e^i_{|\upsilon|+v^i_{c_i-2}},e^i_{|\upsilon|+|\zeta|}),\vec a_{|\upsilon|+|\zeta|})$ holds,
\item for $j>i$, $K^{X\oplus p^j_{|\upsilon|+|\zeta|}}_{R_j;\tau_j}((e^j_{\pi_\upsilon(0)},\ldots,e^j_{\pi_\upsilon(c_j-2)},e^j_{|\upsilon|+|\zeta|}),\vec a_{|\upsilon|+|\zeta|})$.
\end{itemize}
As desired, this yields a process of type $D_{\vec d}$, so we obtain a process of type $D_{(r_1,\ldots,r_k)}$ by induction.
\end{proof}

Combining these as before, we have:
\begin{theorem}
  There is a Turing ideal satisfying $\mathbf{trRT^2_k}$ for all $k$ and $\mathbf{WKL}$ but not $\mathbf{SProdWQO}$.
\end{theorem}

\section{Separating $\mathbf{SCAC}$}

In this section we construct a computable instance $\preceq$ of $\mathbf{SCAC}$ and a Turing ideal $\mathcal{I}$ which has no solution to $\preceq$, but does satisfy both $\mathbf{ProdWQO}$ and $\mathbf{WKL}$.

\begin{definition}
  An \emph{\SCAC} is a \ADSSTS{} $R=(T,\{K_\sigma\}_{\sigma\in T},\{d_\sigma\}_{\sigma\in T})$ with range $\{0,1\}$ and transitive in color $1$.
\end{definition}

Lemmata \ref{thm:ads_sts_diat}, \ref{thm:ads_sts_wkl_gen} and \ref{thm:ads_sts_exists_gen} apply with $J=\{1\}$, $I=\{0,1\}$, so we have:
\begin{lemma}
  If $c$ satisfies all \SCAC{}s in $X$, taking $\prec$ to be the partial ordering so that $a\prec b$ iff $a<b$ and $c(a,b)=1$, whenever $B$ is an $X$-computable infinite set, there exist $a,b,c,d\in B$, $a\prec b$, $c<d$ (so $d\not\prec c$) and $c\not\prec d$.
\end{lemma}
\begin{lemma}
  If $c$ satisfies all \SCAC{}s in $X$ and $U$ is an infinite $X$-computable $\{0,1\}$-branching tree then there is an infinite branch $\Lambda$ so that $c$ satisfies all \SCAC{}s in $X\oplus\Lambda$.
\end{lemma}
\begin{lemma}
  There is a computable stable $c:[\mathbb{N}]^2\rightarrow\{0,1\}$ transitive in the color $1$ satisfying every \SCAC{} in $\emptyset$.
\end{lemma}

In the lemma below, we associate a stable partial ordering $\preceq$ with a coloring with colors $0,1$ so that color $1$ is transitive.  In particular, we say that $\preceq$ satisfies an \SCAC{} when the corresponding coloring does.  So it remains to show:
\begin{lemma}
  Let $\mathcal{I}$ be a countable Turing ideal satisfying $\mathbf{WKL}$, and suppose $\preceq$ satisfies every \SCAC{} in any $X\in\mathcal{I}$ and $c:[\mathbb{N}]^2\rightarrow\{0,1,2\}$ is a coloring in $\mathcal{I}$ with colors $1$ and $2$ transitive.  Then there is an infinite set $S$ so that $c$ restricted to $S$ either omits the color $1$ or omits the color $2$ and $\preceq$ satisfies every \SCAC{} in $X\oplus S$ for any $X\in\mathcal{I}$.
\end{lemma}
\begin{proof}
  The new complication here is that, on the one hand, we have to work with \SCAC{}s, so we have to make sure our construction satisfies the transitivity requirement.  On the other hand, we only have a limited amount of transitivity to work with, because the color $c$ can assign the value $0$.

  Let us say $c$ is \emph{$01$ on $p$} if, for all $x,y\in p$, $c(x,y)\in\{0,1\}$.  Similarly, let us say $c$ is \emph{$02$ on $q$} if, for all $x,y\in q$, $c(x,y)\in\{0,2\}$.

  A \emph{prediction} is a function $\rho:S\rightarrow\{0,1,2\}$ for some set $S$ such that if $a,b\in S$, $a<b$, and $c(a,b)=\rho(b)\neq 0$ then $\rho(a)=\rho (b)$.  We say $\rho\leq \rho '$ if $\dom(\rho)\subseteq\dom(\rho')$ and whenever $\rho(a)\neq 0$, $\rho'(a)=\rho(a)$.  Note that, for any $x$, the function $\rho^x_S$ given by $\rho^x_S(s)=c(s,x)$ is a prediction.

  We work with conditions $(p,q,X)$ such that:
  \begin{itemize}
  \item $c$ is $01$ on $p$,
  \item $c$ is $02$ on $q$,
  \item for all $x,x'\in X$ and all $a\in p$, $c(a,x)=c(a,x')\in\{0,1\}$,
  \item for all $x,x'\in X$ and all $b\in q$, $c(b,x)=c(b,x')\in\{0,2\}$, and
  \item $X$ is an infinite set in $\mathcal{I}$.
  \end{itemize}

  Let us say $(p,q, X)$ \emph{forces $R^+$ on the $01$-side} if whenever $\Lambda$ is an infinite sequence with $p\sqsubseteq\Lambda$, $(\Lambda\setminus p)\subseteq X$, and $c(a,b)\in\{0,1\}$ for $a,b\in\Lambda$, $\prec$ satisfies $R^+$ in $X\oplus\Lambda$.  Similarly, we say $(p,q, X)$ \emph{forces $R^-$ on the $02$-side} if whenever $\Lambda$ is an infinite sequence with $p\sqsubseteq\Lambda$, $(\Lambda\setminus p)\subseteq X$, and $c(a,b)\in\{0,2\}$ for $a,b\in\Lambda$, $\prec$ satisfies $R^-$ in $X\oplus\Lambda$.

  Using Lemma \ref{thm:linearize}, it suffices to show:
\begin{quote}
  $(\ast)$ Suppose $R^+$ are $R^-$ are linear requirements and $(p,q, X)$ is a condition.  Then there is a condition $(p',q', X')$ extending $(p,q, X)$ which either forces $R^+$ on the $01$-side or forces $R^-$ on the $02$-side.
\end{quote}

Let us show $(*)$.  Fix $(p,q, X)$ and requirements $R^+=(T^+,\{L_\sigma\}_{\sigma\in T^+},\{d^+_\sigma\}_{\sigma\in T^+})$ and $R^-=(T^-,\{M_\tau\}_{\tau\in T^-},\{d^-_\tau\}_{\tau\in T^-})$.  Below we always assume elements not in $p,q$ are chosen from $X$.

Suppose we have the fortune to find $p',q'$ extending $p,q$ and witnessing nodes of $T^+$ and $T^-$ so that there is an $a$ such that, for each $x\in p'\setminus p$, $c(a,x)=2$, while for each $y\in q'\setminus q$, $c(a,y)=1$.  Then for any $z>\max\{p',q'\}$, we must either have $c(x,z)\in\{0,1\}$ for all $x\in p'$ or $c(y,z)\in \{0,2\}$ for all $y\in q'$.  Then $p',q'$ function like a split pair: every future $z$ is compatible with either $p'$ or $q'$.  If we could find such pairs consistently, we could carry out a construction like the one in the proof of Theorem \ref{thm:ads_prod_ads}.

So suppose we have a finite set $S$, an extension $p'$ of $p$ witnessing a node of $T^+$, an extension $q'$ of $q$ witnessing a node of $T^-$, and a prediction $\rho$ on $S$ so that for each $x\in p'\setminus p$, $\rho^x_S=\rho$ and for each $y\in q'\setminus q$, $\rho^y_S=\rho$.  Now suppose there are also infinitely many $z$ such that, for some $a\in S$ with $\rho(a)\neq 0$, $c(a,z)\neq \rho (a)$.  Then there must be a single such $a\in S$ with $\rho(a)\neq 0$---without loss of generality, let us assume $\rho (a)=1$---so that there are infinitely many $z$ with $c(a,z)=2$.  Then either we find some $q''$ witnessing a node of $T^-$, putting us in the setting of the previous paragraph, or the set of such $z$ contains no such $q''$ (and therefore forces $R^-$ on the $02$-side).

Our strategy will be to have an ``inner construction'' and an ``outer construction''.  During the inner construction, we will begin constructing extensions of $p$ and $q$ in segments.  Once we have constructed some segments all of whose elements belong to some set $S$, we will look for sequences $p'$ and $q'$ inducing a common prediction $\rho$ on $S$; we then use $p'$ and $q'$ to ``guarantee'' the prediction $\rho$---we divide all $z$ into those with $\rho\leq \rho^z_S$ and those with $\rho\not\leq \rho^z_S$.  When $\rho\leq \rho^z_S$, we continue with the inner construction, using these $z$ to look for segments extending the sequences in $S$.  But if we have many points with $\rho\not\leq \rho^z_S$, we may find either a pair $(p',q'')$ or $(p'',q')$ as in the previous paragraph; in this case we use $(p',q'')$ to extend the outer construction.  When the outer construction extends, we discard all progress on the inner construction and begin a new inner construction.

We first build a tree represneting the outer construction, essentially using the construction of Lemma \ref{thm:ads_prod_ads}, with a minor adjustment---the two halves of our ``split pairs'' will not share blocks of witnesses---and some additional information to account for the inner construction we discuss later.

All the changes are present in the construction of a process of type $\{(1,\ldots,1)\}$: we construct a tree with three non-root nodes, $\langle 0 \rangle$, $\langle 0,1\rangle$, and $\langle 0,2\rangle$.  $K^X_{\{(1,\ldots,1)\},\langle 0\rangle}((b_0),\vec a_0)$ will hold when $b_0=(d_0,e_0,p_0,f_0,q_0)$ and:
\begin{itemize}
\item $p\sqsubset p_0$,
\item $q\sqsubset q_0$,
\item whenever $a<d_0$ and $x,y\in (p_0\setminus p)\cup(q_0\setminus q)$, $c(a,x)=c(a,y)$,
\item $L^{X\oplus p_0}_{\langle 0\rangle}(e_0,\vec a_0)$,
\item $M^{X\oplus q_0}_{\langle 0\rangle}(f_0,\vec a_0)$.
\end{itemize}

$K^X_{\{(1,\ldots,1)\},\langle 0,1\rangle}((b_0,b_1),\vec a_1)$ will hold when $b_0=(d_0,e_0,p_0,f_0,q_0)$ and $b_1=(e_1,p_1)$ and:
\begin{itemize}
\item $p\sqsubseteq p_1$,
\item there is an $a<d_0$ so that for each $x\in (p_1\setminus p)$ and $y\in (q_0\setminus q)$, $c(a,x)=2$ while $c(a,y)=1$,
\item $L^{X\oplus p_1}_{\langle 0\rangle}(e_1,\vec a_1)$.
\end{itemize}
This is, in the node $\langle 0,1\rangle$, the pair $(p_1,q_0)$ form an effective split pair.  

Symmetrically,
$K^X_{\{(1,\ldots,1)\},\langle 0,2\rangle}((b_0,b_1),\vec a_1)$ will hold when $b_0=(d_0,e_0,p_0,f_0,q_0)$ and $b_1=(f_1,q_1)$ and:
\begin{itemize}
\item $q\sqsubseteq q_1$,
\item there is an $a<d_0$ so that for each $x\in (p_0\setminus p)$ and $y\in (q_1\setminus q)$, $c(a,x)=2$ while $c(a,y)=1$,
\item $M^{X\oplus q_1}_{\langle 0\rangle}(f_1,\vec a_1)$.
\end{itemize}

We need to work in two steps---in the first step we find a pair $(p_0,q_0)$, but these don't properly form a split pair because we could easily have neither be extendible.  In the second step we replace one of $p_0$ and $q_0$ with a new sequence---we work with either $(p_1,q_0)$ or $(p_0,q_1)$---because these form a genuine split pair.

We then form compound processes using the same construction as Lemma \ref{thm:ads_prod_ads}, but starting from our new three node basic process.  The result is $R_{(r,s)}=(T_{(r,s)},\{K_{(r,s),\zeta}\},\{d_{(r,s),\zeta}\})$ such that:
 \begin{itemize}
 \item whenever $\zeta\in T_{(r,s)}$ is a leaf, $\Delta^X_{(r,s),\zeta}(c,(b_0,\ldots,b_{|\zeta|-1}),\vec a_0,\ldots,\vec a_{|\zeta|-1})$ implies that each $b_i$ has the form $(d_i,e_i,p_i,f_i,q_i)$, $(e_i,p_i)$, or $(e_i,q_i)$ where:
  \begin{itemize}
  \item $p\sqsubseteq p_i$,
  \item $q\sqsubseteq q_i$,
  \item there are disjoint sequences $(v_0,\ldots,v_{r-1})$ and $(w_0,\ldots,w_{s-1})$ such that, taking $\sigma$ and $\tau$ to be the branches of length $r$ and $s$, respectively,:
    \begin{itemize}
    \item $\Delta^{X\oplus p_i}_{T^+;\sigma}(c,(e_{v_0},\ldots,e_{v_{r-1}}),\vec a_{v_0},\ldots,\vec a_{v_{r-1}})$,
    \item $\Delta^{X\oplus q_i}_{T^-;\tau}(c,(f_{w_0},\ldots,f_{w_{s-1}}),\vec a_{w_0},\ldots,\vec a_{w_{s-1}})$,
    \end{itemize}
  \end{itemize}
\item whenever $\zeta\in T_{(r,s)}$ is not a leaf and $|\zeta|$ is odd, $\Theta_{(r,s),\zeta}(c)$ implies that there is either a $p'$ so that $(p',q, \rho,X)$ forces $R^+$ on the $01$-side, or a $q'$ so that $(p,q', \rho,X)$ forces $R^-$ on the $02$-side, 
\item whenever $\zeta\in T_{(r,s)}$ is not a leaf and $|\zeta|$ is even, $\Theta_{(r,s),\zeta}(c)$ implies that either:
  \begin{itemize}
  \item for every $a<d_{|\zeta|-2}$ there is an $i\in\{1,2\}$ such that there are only finitely many $z>a$ with $c(a,z)=i$,
  \item there is an $a<d_{|\zeta|-2}$ and a $p'$ so that $(p',q, \rho,\{x\in X\mid c(a,x)=2\})$ forces $R^+$ on the $01$-side, or
  \item there is an $a<d_{|\zeta|-2}$ and a $q'$ so that $(p,q', \rho,\{x\in X\mid c(a,x)=1\})$ forces $R^-$ on the $02$-side.
  \end{itemize}
 \end{itemize}

Note that the final case is why $R_{(r,s)}$ does not complete the proof of the lemma: we might indeed be in the case where, for every $a<d_{|\zeta|-2}$ there is an $i\in\{1,2\}$ such that there are only finitely many $z>a$ with $c(a,z)=i$.  So we will need to interpolate additional steps into $R_{(r,s)}$ to account for this possibility.

We will construct our actual requirement $R=(T,\{K_\upsilon\},\{d_\upsilon\})$.  Along with our construction, we define a partial function $f:T\rightarrow T_{(r,s)}$ and, for $\upsilon\in\dom(f)$, a function $\pi_\upsilon:[0,|f(\upsilon)|)\rightarrow[0,\upsilon)$.

We begin by setting $f(\langle\rangle)=\langle\rangle$.

Suppose we have a node $\upsilon$ with $|f(\upsilon)|$ even.  We construct a subtree extending $\upsilon$ as our ``inner construction''.  The precise definition requires some tedious bookkeeping, but the basic idea is that we take a pair $p',q'$ active at the node $f(\upsilon)$ and we search for a set $S$ of points which could be included in extensions of both $p'$ and $q'$, and so that for every prediction $\rho_S$ on $S$, there is either a $p_0$ or a $q_0$ witnessing the corresponding node of $T^+$ or $T^-$.  If we cannot find such an $S$ (despite there being infinitely many points which could be included in such a set), we can use weak K\"onig's lemma to find a partition of points in which there are either no such $p_0$ or no such $q_0$.

If we find such an $S$, we then search for extensions $p'',q''$ of $S$ so that every $x\in (p''\setminus p')\cup (q''\setminus q')$ induces the same prediction $\rho_S$ on $S$.  If we find $p'',q''$, these are associated to a child of $f(\upsilon)$.

Now consider later points $z$.  If $\rho_S\not\leq \rho^z_S$ then we can use $z$ as a point to look for a split pair $(p'',q^*)$ or $(p^*,q'')$; if we find such a pair, we can extend to a descendent of $f(\upsilon)$ and throw away the inner construction.  If we have $\rho_S\leq \rho^z_S$, we begin working towards a new set $S'$ in which we look for either $p_1$ or $q_1$ extending $p_0$ or $q_0$.  If we find such an $S'$, we then search for replacements of $p''$ and $q''$ which also agree on a prediction on $S'$.

The inner construction then iterates this process: we keep extending $p_i$ or $q_i$, and each time we do, we look for new ``guards'' $p''$ and $q''$.  In order to avoid transitivity issues, when we extend $p_i$, we discard our progress on $q_i$.

Formally, in this subtree, enumerate the children of $f(\upsilon)$ as $\zeta_1,\ldots,\zeta_z$.  Each node $\upsilon'$ in this subtree will be associated to tuples $(r^{\upsilon'}_1,s^{\upsilon'}_1,t^{\upsilon'}_i,r^{\upsilon'}_2,s^{\upsilon'}_2,t^{\upsilon'}_i,\ldots,r^{\upsilon'}_z,s^{\upsilon'}_z,t^{\upsilon'}_z)$ with $r^{\upsilon'}_j\leq r$, $s^{\upsilon'}_j\leq s$, and $t^{\upsilon'}_z\in\{0,1\}$, and to sequences of distinct values $v^{\upsilon'}_{j,0},\ldots,v^{\upsilon'}_{j,r_j-1},w^{\upsilon'}_{j,0},\ldots,w^{\upsilon'}_{s_j-1},u_{j}$ where $v^{\upsilon'}_{j,i}<v^{\upsilon'}_{j,i+1}<w^{\upsilon'}_{j,0}$ and $w^{\upsilon'}_{j,i}<w^{\upsilon'}_{j,i+1}<u^{\upsilon'}_{j}<v^{\upsilon'}_{j+1,0}$ and where $u^{\upsilon'}_{j}$ is only present if $r^{\upsilon'}_j+s^{\upsilon'}_j+t^{\upsilon'}_j>0$.  We associate $\upsilon$ with the tuple $(0,\ldots,0)$ (and therefore the sequences $v^{\upsilon'}_{j,i},w^{\upsilon'}_{j,i},u^{\upsilon'}_j$ are empty).  We order the tuples lexicographically and guarantee that the assigned tuples will not decrease along branches of the tree.

We will need to construct three kinds of nodes: the nodes in which we search for sets $S$, the nodes in which we search for guards $p'',q''$ (extending $p'_j,q'_j$), and the nodes in which we search for for the split pair $p^*$ or $q^*$ (also extending $p'_j$ or $q'_j$).  Which of these nodes are children of a given node depends on the parameters so far.  When $t_j=0$, we need to search for a set $S$.  When $t_j=1$, we need to search for guards.  (Further, when $s^{\upsilon'}_j>0$, we will also need to look for alternative, more restrictive guards.)  When $r^{\upsilon'}_j>0$ or $s^{\upsilon'}_j>0$, we must have already found guards, so we also need to look for the corresponding split pair.

Consider a node $\upsilon'$ and we construct the children of $\upsilon'$.

First, consider some $j$ with $t^{\upsilon'}_j=0$.  In this case we will have $b_{|\upsilon'|}=(S_{|\upsilon'|},E_{|\upsilon'|},\lambda_{|\upsilon'|})$.  Let $S^-=\bigcup_{i<r^{\upsilon'}_j}S_{v^{\upsilon'}_{j,i}}\cup\bigcup_{i<s^{\upsilon'}_j}S_{w^{\upsilon'}_{j,i}}$.  If $r^{\upsilon'}_j=s^{\upsilon'}_j=0$, $S^-=\emptyset$ and $\rho_{S^-}$ is trivial; otherwise let $\rho_{S^-}=\rho^x_{S^-}$ for any $x\in p''_{u_j}$.  $\upsilon'$ has a child $\upsilon'{}^\frown\langle (0,j)\rangle$ so that $K^X_{\upsilon'{}^\frown\langle (0,j)\rangle}$ holds when:
\begin{itemize}
\item $S_{|\upsilon'|}$ is a finite set,
\item for every prediction $\rho_S$ on $S^-\cup S_{|\upsilon'|}$ with $\rho_{S^-}\leq \rho_S\upharpoonright S^-$, $\lambda_{|\upsilon'|}(\rho_S)$ is a sequence and $E_{|\upsilon'|}(\rho_S)$ is a number such that either:
  \begin{itemize}
  \item $c$ is $01$ on $\lambda_{|\upsilon'|}(\rho_S)$ and:
    \begin{itemize}
    \item  (if $r^{\upsilon'}_j=0$) $p\sqsubset \lambda_{|\upsilon'|}(\rho_S)$ or (if $r^{\upsilon'}_j>0$) $\lambda_{v^{\upsilon'}_{r_j-1}}(\rho_S\upharpoonright S^-)\sqsubset \lambda_{|\upsilon'|}(\rho_S)$, and
    \item letting $\sigma$ be the branch of length $r^{\upsilon'}_j+1$, $L^{X\oplus \lambda_{|\upsilon'|}(\rho_S)}_{\sigma}((E_{v^{\upsilon'}_{0}}(\rho_S\upharpoonright S_{v^{\upsilon'}_0}),\ldots,E_{v^{\upsilon'}_{r_j-1}}(\rho_S\upharpoonright S_{v^{\upsilon'}_{r_j-1}}),E_{|\upsilon'|}(\rho_S)),\vec a)$,
    \end{itemize}
    or,
  \item $c$ is $02$ on $F_{|\upsilon'|}(\rho_S)$ and:
    \begin{itemize}
    \item (if $s^{\upsilon'}_j=0$) $q\sqsubset \lambda_{|\upsilon'|}(\rho_S)$ or (if $s^{\upsilon'}_j>0$) $\lambda_{w^{\upsilon'}_{s_j-1}}(\rho_S\upharpoonright S^-)\sqsubset \lambda_{|\upsilon'|}(\rho_S)$, and
    \item letting $\tau$ be the branch of length $s^{\upsilon'}_j+1$, $M^{X\oplus \lambda_{|\upsilon'|}(\rho_S)}_{\tau}((E_{v^{\upsilon'}_{0}}(\rho_S\upharpoonright S_{v^{\upsilon'}_0}),\ldots,E_{v^{\upsilon'}_{r_j-1}}(\rho_S\upharpoonright S_{v^{\upsilon'}_{r_j-1}}),E_{|\upsilon'|}(\rho_S)),\vec a)$.
    \end{itemize}
  \end{itemize}
\end{itemize}
We copy the parameters $r_{j'}$, $s_{j'}$ for $j'\leq j$ and reset these to $0$ for $j'>j$; that is:
\begin{itemize}
\item For $j'\leq j$, we set $r^{\upsilon'{}^\frown\langle 0,j\rangle}_{j'}=r^{\upsilon'}_{j'}$, $s^{\upsilon'{}^\frown\langle 0,j\rangle}_{j'}=s^{\upsilon'}_{j'}$, $v^{\upsilon'{}^\frown\langle 0,j\rangle}_{j',i}=v^{\upsilon'}_{j',i}$, and $w^{\upsilon'{}^\frown\langle 0,j\rangle}_{j',i}=w^{\upsilon'}_{j',i}$.
\item For $j'<j$, $t_{j'}^{\upsilon'{}^\frown\langle 0,j\rangle}=t_{j'}^{\upsilon'}$ and $u^{\upsilon'{}^\frown\langle 0,j\rangle}_{j'}=u^{\upsilon'}_{j'}$.
\item $t_j^{\upsilon'{}^\frown\langle 0,j\rangle}=1$ and $u^{\upsilon'{}^\frown\langle 0,j\rangle}_j=|\upsilon'|$. 
\item For $j'>j$, $r^{\upsilon'{}^\frown\langle 0,j\rangle}_{j'}=s^{\upsilon'{}^\frown\langle 0,j\rangle}_{j'}=t^{\upsilon'{}^\frown\langle 0,j\rangle}_{j'}=0$. 
\end{itemize}
We set $d_{\upsilon'^\frown\langle (0,j)\rangle}(i)=d_{\upsilon'}(i)$ for $i<|\upsilon'|$ and $d_{\upsilon'^\frown\langle (0,j)\rangle}(|\upsilon'|)=0$. 

Next, consider some $j$ with $t^{\upsilon'}_j=1$.  Let $S^-=\bigcup_{i<r^{\upsilon'}_j}S_{v^{\upsilon'}_{j,i}}\cup\bigcup_{i<s^{\upsilon'}_j}S_{w^{\upsilon'}_{j,i}}$.  We have two children $\upsilon'{}^\frown\langle (1,j,b)\rangle$ for $b\in\{0,1\}$ with $f(\upsilon'{}^\frown\langle (1,j,b)\rangle)=\zeta_j$.  $K^X_{\upsilon'{}^\frown\langle (1,j,0)\rangle}$ holds when $K^X_{(r,s),\zeta_j}$ holds and there is any $\rho_{S'}\geq \rho^x_{S'}$ (for some, and therefore every, $x\in p_{|\upsilon'|}\cup q_{|\upsilon'|}$) so that $c$ is $02$ on $F_{|\upsilon'|}(\rho_{S'})$.   $K^X_{\upsilon'{}^\frown\langle (1,j,1)\rangle}$ holds when $K^X_{(r,s),\zeta_j}$ holds and for every $\rho_{S'}\geq \rho^x_{S'}$ (for some, and therefore every, $x\in p_{|\upsilon'|}\cup q_{|\upsilon'|}$), $c$ is $01$ on $F_{u_j^{\upsilon'}}(\rho_{S'})$.  

In this case:
\begin{itemize}
\item For $j'<j$, $r^{\upsilon'{}^\frown\langle 1,j,b\rangle}_{j'}=r^{\upsilon'}_{j'}$, $s^{\upsilon'{}^\frown\langle 1,j,b\rangle}_{j'}=s^{\upsilon'}_{j'}$, $v^{\upsilon'{}^\frown\langle 1,j,b\rangle}_{j',i}=v^{\upsilon'}_{j',i}$, $w^{\upsilon'{}^\frown\langle 1,j,b\rangle}_{j',i}=w^{\upsilon'}_{j',i}$, and $u^{\upsilon'{}^\frown\langle 1,j,b\rangle}_{j'}=u^{\upsilon'}_{j'}$.  
\item $r^{\upsilon'{}^\frown\langle 1,j,1\rangle}_j=r^{\upsilon'}_{j}+1$, $v^{\upsilon'{}^\frown\langle 1,j,1\rangle}_{j,i}=v^{\upsilon'}_{j,i}$ for $i<r^{\upsilon'}_j$, and $v^{\upsilon'{}^\frown\langle 1,j,1\rangle}_{j,r^{\upsilon'}_j}=u^{\upsilon'}_j$. 
\item $s^{\upsilon'{}^\frown\langle 1,j,0\rangle}_j=s^{\upsilon'}_j+1$, $w^{\upsilon'{}^\frown\langle 1,j,0\rangle}_{j,i}=w^{\upsilon'}_{j,i}$ for $i<s^{\upsilon'}_j$, and $w^{\upsilon'{}^\frown\langle 1,j,0\rangle}_{j,s^{\upsilon'}_j}=u^{\upsilon'}_j$. 
\item $s^{\upsilon'{}^\frown\langle 1,j,1\rangle}_j=0$ and $u^{\upsilon'{}^\frown\langle 1,j,b\rangle}_j=u^{\upsilon'}_j$.
\item $r^{\upsilon'{}^\frown\langle 1,j,0\rangle}_{j}=r^{\upsilon'}_j$ and $v^{\upsilon'{}^\frown\langle 1,j,0\rangle}_{j,i}=v^{\upsilon'}_{j,i}$. 
\item For $j'>j$, $r^{\upsilon'{}^\frown\langle 1,j,b\rangle}_{j'}=s^{\upsilon'{}^\frown\langle 1,j,b\rangle}_{j'}=s^{\upsilon'{}^\frown\langle 1,j,b\rangle}_{j'}=0$.
\end{itemize}

We set $d_{\upsilon'{}^\frown\langle (1,j,0)\rangle}(w^{\upsilon'{}^\frown\langle (1,j,0)\rangle}_{j,i})=d^-_\tau(i)$ (where $|\tau|=s^{\upsilon'{}^\frown\langle (1,j,0)\rangle}_)$.  We set $d_{\upsilon'{}^\frown\langle (1,j,1)\rangle}(v^{\upsilon'{}^\frown\langle (1,j,1)\rangle})=d^+_\sigma(i)$ (where $|\sigma|=r^{\upsilon'{}^\frown\langle (1,j,1)\rangle}_j$).  Otherwise $d_{\upsilon'{}^\frown\langle (1,j,b)\rangle}(i)=d_{\upsilon'}(i)$ where consistent with transitivity requirements, and as required by transitivity otherwise.

Consider some $j$ with $s^{\upsilon'}_j>0$.  Let $S^-=\bigcup_{i<r^{\upsilon'}_j}S_{v^{\upsilon'}_{j,i}}\cup\bigcup_{i<s^{\upsilon'}_j}S_{w^{\upsilon'}_{j,i}}$.  We have a child $\upsilon'{}^\frown\langle (2,j)\rangle$ with $f(\upsilon'^\frown\langle (2,j)\rangle)=\zeta_j$.  (This case is nearly identical to the $\upsilon'{}^\frown\langle (1,j,1)\rangle$ case above.)  $K^X_{\upsilon'{}^\frown\langle (2,j)\rangle}$ holds when $K^X_{(r,s),\zeta_j}$ holds and for every $\rho_{S'}\geq\rho^x_{S'}$ (for some, and therefore every, $x\in p_{|\upsilon'|}\cup q_{|\upsilon'|}$), $c$ is $01$ on $F_{u_j^{\upsilon'}}(\rho_{S'})$.

\begin{itemize}
\item For $j'<j$, $r^{\upsilon'{}^\frown\langle 2,j\rangle}_{j'}=r^{\upsilon'}_{j'}$, $s^{\upsilon'{}^\frown\langle 2,j\rangle}_{j'}=s^{\upsilon'}_{j'}$, $v^{\upsilon'{}^\frown\langle 2,j\rangle}_{j',i}=v^{\upsilon'}_{j',i}$, $w^{\upsilon'{}^\frown\langle 2,j\rangle}_{j',i}=w^{\upsilon'}_{j',i}$, and $u^{\upsilon'{}^\frown\langle 2,j\rangle}_{j'}=u^{\upsilon'}_{j'}$.
\item $r^{\upsilon'{}^\frown\langle 2,j\rangle}_j=r^{\upsilon'}_{j}+1$, $v^{\upsilon'{}^\frown\langle 2,j\rangle}_{j,i}=v^{\upsilon'}_{j,i}$ for $i<r^{\upsilon'}_j$, and $v^{\upsilon'{}^\frown\langle 2,j\rangle}_{j,r^{\upsilon'}_j}=u^{\upsilon'}_j$. 
\item $s^{\upsilon'{}^\frown\langle 2,j\rangle}_j=0$ and $u^{\upsilon'{}^\frown\langle 2,j\rangle}_j=u^{\upsilon'}_j$.
\item  For $j'>j$, $r^{\upsilon'{}^\frown\langle 2,j\rangle}_{j'}=s^{\upsilon'{}^\frown\langle 2,j\rangle}_{j'}=s^{\upsilon'{}^\frown\langle 2,j\rangle}_{j'}=0$.
\end{itemize}
We set $d_{\upsilon'{}^\frown\langle (1,j,1)\rangle}(v^{\upsilon'{}^\frown\langle (1,j,1)\rangle})=d^+_\sigma(i)$ (where $|\sigma|=r^{\upsilon'{}^\frown\langle (1,j,1)\rangle}_j$).  Otherwise $d_{\upsilon'{}^\frown\langle (1,j,b)\rangle}(i)=d_{\upsilon'}(i)$ where consistent with transitivity requirements, and as required by transitivity otherwise.

For the final case, consider a $j$ with $t^{\upsilon'}_j=0$ and $r^{\upsilon'}_j+s^{\upsilon'}_j>0$.  Then we have two children $\upsilon'{}^\frown\langle (3,j,b)\rangle$ for $b\in\{1,2\}$.  We will set $f(\upsilon'{}^\frown\langle (3,j,b)\rangle)=\zeta_j{}^\frown\langle b\rangle$ and $\pi_{\upsilon'{}^\frown\langle (3,j,b)\rangle}=\pi_\upsilon\cup\{(|f(\upsilon)|,u^{\upsilon'}_j),(|f(\upsilon)|+1,|\upsilon'|)\}$, and $K^X_{\upsilon'{}^\frown\langle (3,j,b)\rangle}$ will hold exactly when $K^X_{\zeta_j{}^\frown\langle b\rangle}$ does.  For $i\leq|\zeta_j|$ we have $d_{\upsilon'{}^\frown\langle (3,j,b)\rangle}(\pi_{\upsilon'{}^\frown\langle (3,j,b)\rangle}(i))=d_{\zeta_j^\frown\langle b\rangle}(i)$; otherwise $d_{\upsilon'{}^\frown\langle (3,j,b)\rangle}(i)=0$ if this is consistent with transitivity, and as required by transitivity otherwise.  These children are not part of the subtree: $|f(\upsilon'{}^\frown\langle (3,j,b)\rangle)|$ is even, so these nodes discard the entire subtree and start a new one.

We must verify that $\Theta^X_\upsilon(c)$ implies the existence of the desired extension of $(p,q,X)$.  Consider some node $\upsilon'$, which belongs to one of our subtrees---$\upsilon'$ extends (perhaps non-properly) a node $\upsilon$ with $|f(\upsilon)|$ even.  

Consider each of the children of $f(\upsilon)$, the nodes $\zeta_j$; when $K^X_{(r,s),\zeta_j}(c,(b_0,\ldots,b_{|\zeta_j|-1}),\vec a_0,\ldots,\vec a_{|\zeta_j|-1})$ holds, we have $b_{|\zeta_j|-1}=(d_{|\zeta_j|-1},e_{|\zeta_j|-1},p_{|\zeta_j|-1},f_{|\zeta_j|-1},q_{|\zeta_j|-1})$ where $p_{|\zeta_j|-1}\sqsupset p_{v_{|\sigma_j|-1}}$ (or $p_{|\zeta_j|-1}\sqsupset p$ if $|\sigma_j|=1$) and $q_{|\zeta_j|-1}\sqsupset q_{w_{|\tau_j|-1}}$ (or $q_{|\zeta_j|-1}\sqsupset q$ if $|\tau_j|=1$) where $\sigma_j,\tau_j$ depend on the node $\zeta_j$.  In particular, we are interested in the segments which these sequences would extend: let $p'_j=p_{\pi_\upsilon(v_{|\sigma_j|}-1)}$ if $|\sigma_j|>1$ and $p'_j=p$ if $|\sigma_j|=1$; similarly, let $q'_j=q_{\pi_\upsilon(w_{|\tau_j|})-1}$ if $|\tau_j|>1$ and $q'_j=q$ if $|\tau_j|=1$.

Each $\zeta_j$ of $f(\upsilon)$ is looking for extensions consisting of points $z$ satisfying conditions like $c(x,z)\in\{0,1\}$ for all $x\in p'_j$ and $c(y,z)\in\{0,2\}$ for all $y\in p'_j$.  These conditions are exhaustive---every point $z$ satisfies these conditions for some $\zeta_j$---so we may choose a $j$ and an $X'\subseteq X$ so that $X'$ is infinite and every $z\in X'$ can be used in an extension witnessing $\zeta_j$.  We may also fix the set $S^-=\bigcup_{i<r^{\upsilon'}_j}S_{v^{\upsilon'}_{j,i}}\cup\bigcup_{i<s^{\upsilon'}_j}S_{w^{\upsilon'}_{j,i}}$.  There is some $\rho_{S^-}$ such that there are infinitely many $z\in X'$ with $\rho^z_{S^-}=\rho_{S^-}$.  Let $X''\subseteq X$ be the set of such $z$.

Suppose $t^{\upsilon'}_j=1$.   Then either $(p,q'_j,X'')$ or $(p'_j,q,X'')$ is the needed condition---if we could find extensions to both $q'_j$ and $p'_j$ in $X''$ then we would have a witness to $\upsilon'{}^\frown\langle (1,j,b)\rangle$ for some $b$ (where $b$ depends on $r_{S^-}$).

So suppose $t_j^{\upsilon'}=0$.  Consider the ``guards'' $p_{u_j^{\upsilon'}},q_{u_j^{\upsilon'}}$ (which extend $p'_j,q'_j$).  For any (equivalently, every) $x\in p_{u_j^{\upsilon'}}\cup q_{u_j^{\upsilon'}}$, let $\rho^0_{S^-}=\rho^x_{S^-}$.  We now consider some cases.
\begin{itemize}
\item Suppose $\rho^0_{S^-}\not\leq\rho_{S^-}$.
  \begin{itemize}
  \item If there is an $a\in S^-$ with $\rho^0_{S^-}(a)=1$ but $\rho_{S^-}(a)=2$ then $(p_{u_j^{\upsilon'}}, q,X'')$ witnesses $\Theta^{X\oplus p_{u_j^{\upsilon'}}}_{R^+;\sigma}(c)$ (where $\sigma$ is the length of the node in $T^+$ corresponding to $\zeta_j$) since any witness to an extension of $\sigma$ would be an extension of $p_{u_j^{\upsilon'}}$ in $X''$, and would therefore witness $\Delta^X_{R;\upsilon'{}^\frown\langle (3,j,1)\rangle}$.
  \item Otherwise there is an $a\in S^-$ with $\rho^0_{S^-}(a)=2$ but $\rho_{S^-}(a)=1$, and $(p,q_{u_j^{\upsilon'}},X'')$ witnesses $\Theta^{X\oplus q_{u_j^{\upsilon'}}}_{R^-;\tau}(c)$ (where $\tau$ is the length of the node in $T^-$ corresponding to $\zeta_j$) since any witness to an extension of $\tau$ would be an extension of $q_{u_j^{\upsilon'}}$ in $X''$, and would therefore witness $\Delta^X_{R;\upsilon'{}^\frown\langle (3,j,2)\rangle}$.
  \end{itemize}
\item Otherwise $\rho^0_{S^-}\leq\rho_{S^-}$.
  \begin{itemize}
  \item Suppose $s^{\upsilon'}_j>0$ but $F_{u_j^{\upsilon'}}(\rho_{S^-})$ is $01$ on $c$.  Then either $(p'_j,q,X'')$ or $(p,q'_j,X'')$ is the needed condition (witnessing $\Theta^{X\oplus p'_j}_{R^+;\sigma}(c)$ or $\Theta^{X\oplus q'_j}_{R^-;\tau}(c)$ respectively)---if we could find extensions to both $p'_j$ and $q'_j$ in $X''$ then we would have a witness to $\Delta^X_{R;\upsilon'{}^\frown\langle (2,j)\rangle}(c)$)
  \item Otherwise $s^{\upsilon'}_j=0$ or $F_{u_j^{\upsilon'}}(r'_{S^-})$ is $02$ on $c$.  Then for every finite set $S>S^-$ with $S\subseteq X''$, there cannot be a function $F$ witnessing $\Delta^X_{R;\upsilon'{}^\frown\langle (0,j)\rangle}(c)$, so there must be some prediction $\rho_S$ such that there is neither a $p^*$ nor a $q^*$ with the needed properties.  In particular, we may divide $S$: let $S_1=\{z\in S\mid \rho_S(z)\in\{0,1\}\}$ and $S_2=\{z\in S\mid \rho_S(z)\in\{0,2\}\}$.  Then there is no $p^*\subseteq S_1$ nor $q^*\subseteq S_2$ with the desired properties.  We now use the technique of Lemma 4.22 of \cite{LST:MR3125903}.  Consider the tree of such partitions.  Since $\mathcal{I}$ satisfies $\mathbf{WKL}$, there is such a partition $X''=X_1\cup X_2$.  One of the pieces $X_1$ or $X_2$ must be infinite, so either $(p_{r_j^{\upsilon'}},q,X_1)$ or $(p,q_{s_j^{\upsilon'}},X_2)$ is the desired condition.
  \end{itemize}
\end{itemize}
\end{proof}

Combining these as before, we have:
\begin{theorem}
  There is a Turing ideal satisfying $\mathbf{ProdWQO}$ and $\mathbf{WKL}$ but not $\mathbf{SCAC}$.
\end{theorem}

\section{A Question}

The original goal of this project was simply to separate $\mathbf{ProdWQO}$ from $\mathbf{SCAC}$; incorporating $\mathbf{WKL}$ (and therefore simultaneously separating $\mathbf{ADS}+\mathbf{WKL}$ from $\mathbf{SCAC}$) seemed to be forced on the project by the nature of the arguments needed.

\begin{question}
  Is it possible to separate $\mathbf{ProdWQO}$ from $\mathbf{SCAC}$ without separating $\mathbf{ProdWQO}+\mathbf{WKL}$ from $\mathbf{SCAC}$?
\end{question}

In particular, it would be interesting to identify a way to make precise the claim that the separation of $\mathbf{ProdWQO}$ from $\mathbf{SCAC}$ somehow requires dealing with $\mathbf{WKL}$.

\printbibliography
\end{document}